\documentclass{article}
\usepackage{amsmath,amsthm,amssymb,mathtools,amscd }
\usepackage[matrix,arrow]{xy} 
\newtheorem{thm}{Theorem}[section]
\newtheorem{defi}[thm]{Definition}
\newtheorem{lem}[thm]{Lemma}
\newtheorem{prop}[thm]{Proposition}
\newtheorem{cor}[thm]{Corollary}

\newtheorem{remark}[thm]{Remark}
\newtheorem*{thmA}{Theorem A}
\newtheorem*{thmB}{Theorem B}
\newtheorem*{thmC}{Theorem C}
\newtheorem*{thmD}{Theorem D}
\newcommand{\g}{\mathfrak{g}}
\newcommand{\h}{\mathfrak{h}}
\newcommand{\De}{\Delta}
\newcommand{\al}{\alpha}
\newcommand{\Cg}{\mathfrak{Cg}}
\newcommand{\bo}{\mathfrak{b}}
\newcommand{\n}{\mathfrak{n}}
\newcommand{\D}{\mathbb{D}}

\newcommand{\gr}{\mathrm{Gr}^{\lambda}D}
\newcommand{\vl}{|\lambda\rangle}
\newcommand{\C}{\mathbb{C}}
\newcommand{\df}{\mathcal{D}}

\newcommand{\sll}{\mathfrak{sl}_{2l+1}}

\newcommand{\Z}{\mathbb{Z}}
\makeatletter
\@addtoreset{equation}{section}

\makeatother
\makeatletter

\@addtoreset{equation}{section}
\makeatother

\title{Demazure slices of type $A_{2l}^{(2)}$}
\author{Masahiro Chihara\thanks{Department of Mathematics, Kyoto University, Oiwake Kita-Shirakawa Sakyo Kyoto 606-
8502 JAPAN, E-mail:chihara@math.kyoto-u.ac.jp}}
\date{}
\begin{document}
\maketitle
\begin{abstract}
We consider a Demazure slice of type $A_{2l}^{(2)}$, that is an associated graded piece of an infinite-dimensional version of a Demazure module. We show that a global Weyl module of a hyperspecial current algebra of type $A_{2l}^{(2)}$ is filtered by Demazure slices. We calculate extensions between a Demazure slice and a usual Demazure module and prove that a graded character of a Demazure slice is equal to a nonsymmetric Macdonald-Koornwinder polynomial divided by its norm. In the last section, we prove that a global Weyl module of the special current algebra of type $A_{2l}^{(2)}$ is a free module over the polynomial ring arising as the endomorphism ring of itself.
\end{abstract}
\section*{Introduction}
A Demazure module in a highest weight module $L(\Lambda)$ of a Kac-Moody Lie algebra $\g$ is studied for a long time. For an affine Lie algebra $\g$, there are two types of Demazure modules in the literature \cite{Kas, Kum}. One is a thin Demazure module, that is usual Demazure module. The other is a thick Demazure module, that is an infinite-dimensional version of a thin Demazure module. Consider an affine Lie algebra of type $X_{l}^{(r)}$ ($X=A,D,E$) and $r=1,2,3$ that is called type I in \cite{CI}. Its level one thin Demazure module has special features. Sanderson \cite{San} and Ion \cite{Ion} showed that its graded character is equal to a nonsymmetric Macdonald polynomial specialized at $t=0$ in $X_{l}^{(r)}\neq A_{2l}^{(2)}$-case and equal to a nonsymmetric Macdonald-Koornwinder polynomial specialized at $t=0$ in $A_{2l}^{(2)}$-case. Another special feature is the connection with a local Weyl module of a current algebra $\Cg$ that is a hyperspecial maximal parabolic subalgebra of $\g$ (\cite{CIK}). Chari-Loktev \cite{CL}, Fourier-Littelmann \cite{CL}, Fourier-Kus \cite{FK} and Chari-Ion-Kus \cite{CIK} showed that a $\Cg$-stable level one thin Demazure module is isomorphic to a local Weyl module as a $\Cg$-module.

Less is known about a thick Demazure module compared to a thin Demazure module. A thick Demazure module is a module of a lower Borel subalgebra that is generated from an extremal weight vector  of $L(\Lambda)$. Cherednik and Kato \cite{CK} recently studied a Demazure slice that is defined as a quotient module of a thick Demazure module. In type I but not of type $A_{2l}^{(2)}$, they showed that a global Weyl module of $\Cg$ have a filtration by level 1 Demazure slices. Moreover they calculated extensions between a level one Demazure slice and a level one thin Demazure module. As a result, they showed graded characters of a level one Demazure slice and a thin Demazure module are orthogonal to each other with respect to the Euler-Poincar{\'e}-pairing. In particular, the graded character of a Demazure slice is equal to a nonsymmetric Macdonald polynomial specialized at $t=\infty$ divided by its norm. 

In this paper, we provide anaologues of these results in \cite{CK} for $A_{2l}^{(2)}$.\\
Let $\g$ be an affine Kac-Moody Lie algebra of type $A_{2l}^{(2)}$ and $\mathring{\g}$ be a simple Lie algebra of type $C_{l}$ contained in $\g$. Let $\h$ be a Cartan subalgebra of $\g$. Let $\mathring{P}$ be the integral weight lattice of $\mathring{\g}$ and $\mathring{P}_+$ be the set of dominant integral weights of $\mathring{\g}$. For each $\lambda\in\mathring{P}_+$, we have a $\Cg$-module $W(\lambda)$, that is called a global Weyl module. Level one Demazure slices and thin Demazure modules are parametrized by $\lambda\in\mathring{P}$ as $\D^{\lambda}$ and $D_{\lambda}$, respectively. Let $\Lambda_{0}$ be the unique level one dominant integral weight of $\g$ and let $\delta$ be the simple imaginary root of $\g$. Let $\mathring{W}$ be the Weyl group of $\mathring{\g}$. Let $\bo_-$ be a lower-triangular Borel subalgebra of $\g.$
\begin{thmA}[=Theorem \ref{realization}]
For each $\lambda\in \mathring{P}_+,$ the global Weyl module $W(\lambda)\otimes_{\C}\C_{\Lambda_0}$ has a  filtration by Demazure slices as $\bo_-$-module and each $\D^{\mu}$ $(\mu\in \mathring{W}\lambda)$ appears exactly once.
\end{thmA}
 Let $\mathfrak{B}$ be a full subcategory of the category of $U(\bo_-)$-modules and $\langle-,-\rangle_{\mathrm{Ext}}$ be the Euler-Poincar\'e-pairing associated to $\mathrm{Ext}_{\mathfrak{B}}$ (see Section 1 for their precise definitions).
\begin{thmB}[=Thorem \ref{extduality}]
For each $\lambda$, $\mu\in \mathring{P}$, $m\in\mathbb{Z}/2$ and $k\in \mathbb{Z}$, we have
$$\mathrm{dim}_{\mathbb{C}}\;\mathrm{Ext}_{\mathfrak{B}}^n(\D^{\lambda}\otimes_{\C}\C_{m\delta+k\Lambda_0},D_{\mu}^{\vee})=\delta_{n,0}\delta_{m,0}\delta_{k,0}\delta_{\lambda,\mu}\;\;\;n\in\Z_+,$$ where $\vee$ means the restricted dual.
\end{thmB}
For each $\lambda\in \mathring{P}$, let $\bar{E}_{\lambda}(x_1,..,x_l,q)$ and $E^{\dag}_{\lambda}(x_1,...,x_l,q)$ be nonsymmetric Macdonald polynomials specialized at $t=0,\;\infty$ respectively. Let $(-,-)$ be the Weyl group invariant inner product on the dual of a Cartan subalgebra $\h^*$ normalized so that the square length of the shortest roots of $\g$ with respect to $(-,-)$ is $1$. Let $\mathrm{gch}\;M$ be a graded character of $M$ (see \S$\ref{representation}$ for the definition). As a corollary of Theorem B, we have
\begin{thmC}[=Corollary \ref{character of demazure slice}]
For each $\lambda\in\mathring{P}$, we have
$$\mathrm{gch}\;\D^{\lambda}=q^{\frac{(b|b)}{2}}{E}^{\dag}_{\lambda}(x_1^{-1},...,x_l^{-1},q^{-1})/\langle \bar{E}_{\lambda},E^{\dag}_{\lambda}\rangle_{\mathrm{Ext}}.$$
\end{thmC}
In this paper, we refer to a maximal parabolic subalgebra of affine Lie algebra that contains a finite dimensional simple Lie algebra as a current algebra. For an affine Lie algebra of type $A_{2l}^{(2)}$, two kind of current algebras are studied in the literature. They contain simple Lie algebras of type $C_l$ and $B_l$, respectively. The former is called a hyperspecial current algebra. A dimension formula of a local Weyl module of a hyperspecial current algebra and freeness of a global Weyl module over its endomorphism ring are proved in \cite{CIK}. The latter is  called a special current algebra and a dimension formula of a local Weyl module of a special current algebra is proved in \cite{FK} and \cite{FM}. Let $\Cg^{\dag}$ be a special current algebra of $\g$. Then $\Cg^{\dag}$ contains a simple Lie algebra $\mathring{\g}^{\dag}$ of type $B_l$. Let $W(\lambda)^{\dag}$ be a global Weyl module of $\Cg^{\dag}$. Let $\Cg^{\dag\prime}=[\Cg^{\dag},\Cg^{\dag}]$. In the last section, we prove the following theorem.
\begin{thmD}[=Theorem \ref{freeness}+Theorem \ref{polynomialring}]
Let $\lambda$ be a dominant integral weight of $\mathring{\g}^{\dag}$. The endomorphism ring $\mathrm{End}_{\Cg^{\dag\prime}}(W(\lambda)^{\dag})$ is a polynomial ring and $W(\lambda)^{\dag}$ is free over $\mathrm{End}_{\Cg^{\dag\prime}}(W(\lambda)^{\dag})$.
\end{thmD}  The organization of the paper is as follows: In section one, we prepare basic notation and definitions. Section two is about a Demazure slice. Main contents of section two are the relation between  a global Weyl module and a Demazure slice (Theorem A), and calculation of extensions between a Demazure slice and a thin Demazure module (Theorem B). As a corollary, we prove a character formula of a Demazure slice (Theorem C). In section three, we study a global Weyl module of a special current algebra of type $A_{2l}^{(2)}$. We prove the endomorphism ring of a global Weyl module is isomorphic to a polynomial ring and a global Weyl module is free over its endomorphism ring (Theorem D).  
\subsubsection*{Acknowledgement}The author thanks Syu Kato and Ievgen Makedonskyi for much advice and discussion.
\section{Preliminaries}
We refer to \cite{Sahi2000}, \cite[Chapter 6]{Kac} and \cite{CI} for general terminologies throughout this section. Mainly we refer to \cite{Kac} for \S\ref{Affine Kac-Moody algebra of type $A_{2l}^{(2)}$} and \S\ref{section weyl group} and refer to \cite{CI} for the \S\ref{Hyperspecial current algebra of $A_{2l}^{(2)}$}.
\subsection{Notations}
We denote the set of complex numbers by $\mathbb{C}$, the set of integers by $\mathbb{Z}$, the set of nonnegative integers by $\mathbb{Z}_{+}$, the set of rational numbers by $\mathbb{Q}$, and the set of natural numbers by $\mathbb{N}$. We work over the field of complex numbers. In particular, a vector space is a $\mathbb{C}$-vector space. For each $x\in \mathbb{Q}$, we set $\lfloor x \rfloor:=\mathrm{max}\{z\in \mathbb{Z}|\;x\geq z\}$. We set $x^{(r)}:=x^{r}/r!$ for an element $x$ of a $\C$ algebra.

\subsection{Affine Kac-Moody algebra of type $A_{2l}^{(2)}$}\label{Affine Kac-Moody algebra of type $A_{2l}^{(2)}$}
Let $\g$ be an affine Kac-Moody algebra of type $A_{2l}^{(2)}$ and  $\h$ be its Cartan subalgebra. We denote the set of roots of $\g$ with respect to $\h$ by $\De$ and fix a set of simple roots $\{ \al_0,\al_1,...,\al_l\}$, where $\al_0$ is the shortest simple root of $\g$. Let $\De_{+}$ and $\De_-$ be the set of positive and negative roots, respectively. We set the simple imaginary root as $\delta:=2\al_0+\al_1+\cdots+\al_l$, the set of imaginary roots as $\De_{im}:=\mathbb{Z}\delta$, and the set of real roots $\De_{re}:=\De\backslash \De_{im}$. We set $Q:=\bigoplus_{i=0}^{l}\mathbb{Z}\al_i,$ $\mathring{Q}:=\bigoplus_{i=1}^{l}\mathbb{Z}\al_i,$ and $\mathring{Q}^{\dag}:=\bigoplus_{i=0}^{l-1}\mathbb{Z}\al_i.$ We set $Q_+:=\bigoplus_{i=0}^{l}\mathbb{Z}_+\al_i$, $\mathring{Q}_+:=\bigoplus_{i=1}^{l}\mathbb{Z}_+\al_i$, and $\mathring{Q}^{\dag}_+:=\bigoplus_{i=0}^{l-1}\mathbb{Z}_+\al_i$. Let $\mathring{\De}=\De\cap \mathring{Q}$. The set $\mathring{\De}$ is a root system of type $C_{l}$. Using the standard basis $\varepsilon_1,...,\varepsilon_l$ of $\mathbb{R}^{l}$, we have:
$$\mathring{\De}=\{\pm(\varepsilon_{i}\pm\varepsilon_{j}),\;\pm2\varepsilon_{i}|\;i,j=1,...,l\}.$$ We denote the set of short roots of $\mathring{\g}$ by $\mathring{\De}_{s}$ and the set of long roots of $\mathring{\g}$ by $\mathring{\De}_l.$ We have $$\De_{re}=(\mathring{\De}_s+\mathbb{Z}\delta)\cup(\mathring{\De}_l+2\mathbb{Z}\delta)\cup\frac{1}{2}(\mathring{\De}_l+(2\mathbb{Z}+1)\delta)$$ and $$\al_0=\delta+\varepsilon_1,\;\;\al_1=-\varepsilon_1+\varepsilon_2,\;\cdots, \;\al_{l-1}=-\varepsilon_{l-1}+\varepsilon_l,\;\;\al_l=-2\varepsilon_l.$$ We set $\De_{l\pm}:=\De_{\pm}\cap \De_{l}$, $\De_{s\pm}:=\De_{\pm}\cap \De_{s}$ and $\mathring{\De}_{\pm}:=\De_{\pm}\cap \mathring{\De}$. For each $\al\in\De_{re}$, let $\check{\al}\in \h$ be the corresponding coroot of $\g$. Let $\theta$ be the highest root of $\mathring{\De}.$ Let $d\in \h$ be the scaling element that satisfies $\al_i(d)=\delta_{i,0}$. We denote a central element of $\g$ by $K=\check{\al}_0+2\check{\al}_1+\cdots+2\check{\al}_l.$ For each $\al\in \De$, we denote the root space corresponding to $\al$ by $\g_{\al}.$ For each $\al\in \De_{re},$ the root space $\g_{\al}$ is one dimensional and we denote a nonzero vector in $\g_{\al}$ by $e_{\al}.$ A Borel subalgebra $\bo_{\pm}$ and a maximal nilpotent subalgebra $\n_{\pm}$ of $\g$ are $$\bo_{+}=\h\oplus\n_{+},\;\n_{+}=\underset{\al \in \De_{+}}{\bigoplus}\g_{\al},\;\bo_{-}=\h\oplus\n_{-},\;\rm{and}\;\n_{-}=\underset{\al \in \De_{-}}{\bigoplus}\g_{\al}.$$ For each $i\in\{0,1,...,l\}$, we define $\Lambda_i \in \h^{*}$ by $$\Lambda_{i}(\check{\al}_j)=\delta_{i,j},\;\Lambda_{i}(d)=0.$$ We set $$P:=\mathbb{Z}\Lambda_0\oplus\cdots \oplus\mathbb{Z}\Lambda_l\oplus\frac{\mathbb{Z}}{2}\delta,\;\;\mathrm{and}\;\;P_+:=\mathbb{Z}_+\Lambda_0\oplus\cdots \oplus\mathbb{Z}_+\Lambda_l\oplus\frac{\mathbb{Z}}{2}\delta.$$ We set $\varpi_{i}:=\Lambda_i-2\Lambda_{0}$ ($i\in\{1,...,l\}\rm{)},$ $$\mathring{P}=\mathbb{Z}\varpi_{1}\oplus\cdots \oplus\mathbb{Z}\varpi_{l}\;\;\mathrm{and}\;\; \mathring{P}_+=\mathbb{Z}_+\varpi_{1}\oplus\cdots \oplus\mathbb{Z}_+\varpi_{l}.$$ We set $\mathring{Q}^{\prime}:=\mathring{Q}+\frac{\mathbb{Z}}{2}\mathring{\De}_{l}$ and $\mathring{Q}_{+}^{\prime}:=\mathring{Q}_{+}+\frac{\mathbb{Z}_+}{2}\mathring{\De}_{l+}.$

\subsection{Hyperspecial current algebra of $A_{2l}^{(2)}$}\label{Hyperspecial current algebra of $A_{2l}^{(2)}$}
We set $\mathring{\h}:=\bigoplus_{i=1}^{l}\mathbb{C}\al_i,$ $\mathring{\g}:=\underset{\al\in \mathring{\De}}{\bigoplus}\g_{\al}\oplus\mathring{\h}$, and $\mathring{\bo}_{+}:=\underset{\al\in \mathring{\De}_+}{\bigoplus}\g_{\al}.$ Then $\mathring{\g}$ is a finite dimensional simple Lie algebra of type $C_l$, the Lie subalgebra $\mathring{\h}$ is a Cartan subalgebra of $\mathring{\g}$, the Lie subalgebra $\mathring{\bo}_+$ is a Borel subalgebra of $\mathring{\g},$ and $\mathring{\De}$ is the set of roots of $\mathring{\g}$ with respect to $\mathring{\h}.$ The lattice $\mathring{P}$ is the integral weight lattice of $\mathring{\g},$ and $\mathring{P}_+$ is the set of dominant integral weight of $\mathring{\g}.$ A hyperspecial current algebra $\Cg$ is a maximal parabolic subalgebra of $\g$ that contains $\mathring{\g}$. I.e. 
$$\Cg:=\mathring{\g}+\bo_{-}.$$ We set $\Cg^{\prime}:=[\Cg,\Cg].$
\begin{remark}
Usually $\Cg^{\prime}$ is called current algebra in the literature. We have $\Cg=\Cg^{\prime}\oplus\C d\oplus \C K.$
\end{remark}
We define a subalgebra $\Cg_{im}$ of $\Cg$ by 
$$\Cg_{im}:=\underset{n\in-\mathbb{N}}{\bigoplus} \g_{n\delta},$$ and define a subalgebra $\mathfrak{Cn}_+$ of $\Cg$ by $$\mathfrak{Cn}_+:=\underset{\al\in (\mathring{\De}_{s+}-\mathbb{Z}_+\delta)\cup (\mathring{\De}_{l+}-2\mathbb{Z}_+\delta)\cap \frac{1}{2}(\mathring{\De}_{l+}-(2\mathbb{Z}_++1)\delta)}{\bigoplus}\g_{\al}.$$

\subsection{Weyl group}\label{section weyl group}
Let $s_{\al}\in\; \rm{Aut}\;(\h^*)$ be the simple reflection corresponding to $\al\in \De_{re}$. We have
$$s_{\al}(\lambda)=\lambda-\langle\lambda,\; \check{\al}\rangle\al,\;\rm{for}\; \lambda\in \h^*.$$
We set $W$ as the subgroup of $\rm{Aut}\;(\h^*)$ generated by $s_{\al}\;(\al\in \De_{re})$, and $\mathring{W}$ as the subgroup generated by $s_{\al}\;(\al\in \mathring{\De})$. For each $i=0,...,l$, let $s_i:=s_{\al_i}$. Then $W$ is generated by $s_i$ ($i=0,...,l$), and $\mathring{W}$ is generated by $s_i$ ($i=1,...,l$). Let $(-|-)$ be a $W$-invariant bilinear form on $\h^*$ normalized so that $(\al_0|\al_0)=1.$ For each $\mu\in \mathring{P}$, we define $t_{\mu}\in \rm{Aut}\;(\h^*)$ by $$t_{\mu}(\lambda)=\lambda+\langle\lambda,K\rangle\mu-((\lambda|\mu)+\frac{1}{2}(\mu|\mu)\langle\lambda,K\rangle)\delta.$$ We have $t_{\mu}\in W$ and \begin{equation}\label{weyl group}
W=\mathring{W}\ltimes \mathring{P}.
\end{equation}
For each $\lambda\in \mathring{P}$, we denote the unique element of $\mathring{W}\lambda\cap \pm P_{+}$ by $\lambda_{\pm}$, respectively. We set $\rho:=\frac{1}{2}\sum_{\al\in \mathring{\De}_+}\al$. For each $w\in \mathring{W}$ and $\lambda\in \mathring{P}$, we define $w\circ \lambda:=w(\lambda+\rho)-\rho$. For each $\Lambda\in P$, we set $W_{\Lambda}:=\{w\in W|\;w\Lambda=\Lambda\}.$ We denote the set of minimal coset representatives of $\mathring{W}\backslash W$ by $W_{0}.$

\begin{defi}[Reduced expression]Each $w\in W$ can be written as a product $w=s_{i_1}s_{i_2}\cdots s_{i_n}$ $(i_j\in\{0,...,l\})$. If $n$ is minimal among such expressions, then $s_{i_1}s_{i_2}\cdots s_{i_n}$ is called a reduced expression of $w$ and $n$ is called the length of $w$ $($written as $l(w))$.
\end{defi}
\begin{defi}[Left weak Bruhat order]\label{Left weak Bruhat order}Let $w\in W$ and $i=0,..,l$. We write $s_iw>w$ if $l(s_iw)>l(w)$ holds. Left weak Bruhat order is the partial order on $W$ generated by $>$.
\end{defi}
\begin{defi}[Macdonald order]We write $\mu \succeq \lambda$ if and only if one of the following two conditions holds: \item$\mathrm{(1)}$ $\mu-\lambda\in \mathring{Q}_+$ if $\mu \in \mathring{W}\lambda$;
\item $\mathrm{(2)}$ $\lambda_{+}-\mu_{+}\in \mathring{Q}^{\prime}_{+}$ if $\mu_+\neq\lambda_+$.
\end{defi}
For $w\in W$ and $\mu\in\mathring{P}$, let $w(\!(\mu)\!)\in \mathring{P}$ be the restriction of $w(\mu+\Lambda_0)$ to $\mathring{\h}$. For each $\lambda\in \mathring{P}$, let $\pi_{\lambda}\in W$ be a minimal  length element such that $(\pi_{\lambda}\Lambda_{0})(\!(0)\!)=\lambda$. For each $\mu\in \mathring{P},$ we denote the convex hull of $\mathring{W}\mu$ by $C(\mu)$.

\begin{lem}[\cite{Mac} Proposition 2.6.2]\label{macdonald}
If $\mu \in \mathring{P}$, then $C(\mu)\cap (\mu +\mathring{Q}^{\prime})\subseteq \bigcap_{w\in \mathring{W}}w(\mu_+-\mathring{Q}_{+}^{\prime}).$
\end{lem}
\begin{proof}
The set $w(\mu_+-\mathring{Q}_{+}^{\prime})$ is the intersection of $\mu_++\mathring{Q}^{\prime}$ with the convex hull of $w(\mu_+-\mathring{Q}_{+}^{\prime})$. The set $\mathring{W}\mu$ is contained in $\bigcap_{w\in \mathring{W}}w(\mu_+-\mathring{Q}_{+}^{\prime}).$ Hence we have $C(\mu)\cap (\mu +\mathring{Q}^{\prime}) \subseteq \bigcap_{w\in \mathring{W}}w(\mu_+-\mathring{Q}_{+}^{\prime})$
\end{proof}
\begin{lem}\label{cherednik order and bruhat order}
\item$\mathrm{(1)}$ If $w>v\in W$, then $v(\!(0)\!)\succeq w(\!(0)\!)$;
\item$\mathrm{(2)}$ Let $b$, $c\in \mathring{P}$ satisfy $b=s_i(\!(c)\!)$ for some $i=0,...,l$. Then$$c\succ b\iff \pi_b=s_i\pi_c>\pi_c.$$
\end{lem}
\begin{proof}
First, we prove (1). It is enough to prove the assertion for $w=s_iv$. Since $w>v$, we have $\langle v\Lambda_{0}, \check{\al}_i\rangle \geq 0$. This implies $v\Lambda_{0}-w\Lambda_{0}\in Q_{+}$. Hence we have $v(\!(0)\!)\succeq w(\!(0)\!)$ if $i\neq 0$. If $i=0$, then we have $w(\!(0)\!)-v(\!(0)\!)=\langle v\Lambda_0,\check{\al}_{0}\rangle\theta/2$. We set $N=\langle v\Lambda_0, \check{\al}_0\rangle.$ We have $\langle w(\!(0)\!),\check{\theta} \rangle=(N+1)/2.$ Hence $s_{\theta}(w(\!(0)\!))=w(\!(0)\!)-\frac{N+1}{2}\theta$ and $v(\!(0)\!)=\frac{N}{N+1}s_{\theta}(w(\!(0)\!))+\frac{1}{N+1}w(\!(0)\!)$. Therefore, $v(\!(0)\!)\in C(w(\!(0)\!))\cap (w(\!(0)\!)+\mathring{Q}^{\prime}).$ By Lemma \ref{macdonald}, $w(\!(0)\!)_+-v(\!(0)\!)_+\in \mathring{Q}^{\prime}_+.$ Hence $v(\!(0)\!)\succeq w(\!(0)\!).$ \\Next, we prove (2). We already proved ($\Leftarrow$). So we prove ($\Rightarrow$). By Definition \ref{Left weak Bruhat order}, we have $s_i\pi_c>\pi_c$ or $s_i\pi_c<\pi_c$. From $c\succ b$ and (1), we have $s_i\pi_c>\pi_c$ and $\pi_b>s_i\pi_b.$ We have $(s_i\pi_c)(\!(0)\!)=b$ thanks to $b=s_i(\!(c)\!)$. We show that $\pi_b=s_i\pi_c$. If $\pi_b\neq s_i\pi_c$, then we have $l(s_i\pi_c)>l(\pi_b)$ by the minimality of $l(\pi_{b})$. Since $l(\pi_b)=l(s_i\pi_b)+1$, $l(s_i\pi_c)=l(\pi_c)+1$ and $l(s_i\pi_c)>l(\pi_b)$, we get $l(\pi_c)>l(s_i\pi_b)$. This contradicts the minimality of $l(\pi_c)$. Hence the assertion follows.
\end{proof}

\subsection{Macdonald-Koornwinder polynomials}
In this subsection, we recall materials presented in \cite[\S3]{Sahi2000} and \cite{Ion}, and we specialize parameters $t,$ $t_0,$ $u_0,$ $t_l,$ $u_l$ in \cite{Sahi2000} as $t_0=t_l=u_0=t$ and $u_l=1$ \cite{Ion}.
\subsubsection{Nonsymmetric case}
We set $\mathbb{F}:=\mathbb{Q}(\!(t,q^{1/2})\!)$. Let $\mathbb{F}[\mathring{P}]$ be a group ring of $\mathring{P}$ over $\mathbb{F}$ and $X^{\lambda}$ be an element of $\mathbb{F}[\mathring{P}]$ corresponding to $\lambda\in \mathring{P}.$ We identify $\mathbb{F}[x_1^{\pm1},...,x_l^{\pm1}]$ with  $\mathbb{F}[\mathring{P}]$ by $x_i=X^{\varepsilon_i}$ for each $i\in\{1,...,l\}$. We define $$\De(x):=\De(x)_+\De(x^{-1})_+\prod_{n\in \mathbb{N}}(1-q^n)^l\in \mathbb{F}[\![x_1^{\pm1},...,x_l^{\pm1}]\!]$$ by
$$\De(x)_+:=\underset{i=1,...,l}{\prod}\frac{(x_i)_{\infty}(-x_i)_{\infty}(q^{1/2}x_i)_{\infty}}{(tx_i)_{\infty}(-tx_i)_{\infty}(q^{1/2}t^2x_i)_{\infty}}\underset{1\leq i<j\leq l}{\prod}\frac{(x_ix_j)_{\infty}(x_ix_{j}^{-1})_{\infty}}{(tx_ix_j)_{\infty}(tx_ix_{j}^{-1})_{\infty}}.$$
Here $(u)_{\infty}=\underset{n\in\mathbb{Z}_{+}}{\prod}(1-q^nu)$.
We define $$\varphi(x):=\underset{i=1,...,l}{\prod}\frac{(x_i-t)(x_i+t)}{x_i^2-1}\underset{1\leq i<j\leq l}{\prod}\frac{(x_ix_j-t)(x_ix_j^{-1}-t)}{(x_ix_j-1)(x_ix_j^{-1}-1)}$$
and $\mathcal{C}(x):=\De(x)\varphi(x)$.
We have $$\De(x)_+|_{t=0}=\underset{i=1,...,l}{\prod}(x_i)_{\infty}(-x_i)_{\infty}(q^{1/2}x_i)_{\infty}\underset{1\leq i<j\leq l}{\prod}(x_ix_j)_{\infty}(x_ix_{j}^{-1})_{\infty}$$ and $$\varphi(x)|_{t=0}=\underset{i=1,...,l}{\prod}\frac{1}{1-x_i^{-2}}\underset{1\leq i<j\leq l}{\prod}\frac{1}{(1-x^{-1}_ix^{-1}_j)(1-x_i^{-1}x_j)}.$$ Under the identification $x_i=X^{\varepsilon_i}$, we have $$\De(x)|_{t=0}=\underset{\al \in \De \;\rm{and}\; \al(d)\leq 0}{\prod}(1-X^{\al})^{\mathrm{dim}\;\g_{\al}}\;\;\mathrm{and}\;\;\varphi(x)|_{t=0}=\underset{\al\in\mathring{\De}_+}{\prod}\frac{1}{1-X^{\al}}.$$ Hence we have $$\mathcal{C}|_{t=0}=\underset{\al \in \De_-}{\prod}(1-X^{\al})^{\mathrm{dim}\;\g_{\al}}.$$
\begin{defi}
We define an inner product on $\mathbb{F}[x_1^{\pm1},...,x_l^{\pm1}]$ by 
$$\langle f,g\rangle^{\prime}_{nonsym}:=\text{the constant term of}\; fg^{\star}\mathcal{C}\in \mathbb{F}.$$
Here $\star$ is the involution on $\mathbb{F}[x_1^{\pm1},...,x_l^{\pm1}]$ such that $q^{\star}=q^{-1}$, $x_i^{\star}=x_i^{-1}$ and $t^{\star}=t^{-1}$.
\end{defi}
\begin{defi}
The set of nonsymmetric Macdonald-Koornwinder polynomials $\{E_{\lambda}(x,q,t,a,b,c,d)\}_{\lambda\in \mathring{P}}$ is a collection of elements in $\mathbb{F}[\mathring{P}]$ indexed by $\mathring{P}$ with the following properties:
\item $\mathrm{(1)}$ $\langle E_{\lambda},E_{\mu}\rangle^{\prime}_{nonsym}=0$ if $\lambda \neq \mu$;
\item $\mathrm{(2)}$ $E_{\lambda}=X^{\lambda}+\underset{\mu\succ \lambda}{\sum}c_{\mu}X^{\mu}$.
\end{defi}
As in \cite[\S3.2]{Ion}, we set 
$$\bar{E}_{\lambda}:=\underset{t\to 0}{\lim}\;E_{\lambda},\;E^{\dag}:=\underset{t\to 0}{\lim}\;E^{\star}_{\lambda}.$$
Let $\langle-,-\rangle_{nonsym}$ be a specialization of $\langle-,-\rangle_{nonsym}^{\prime}$ at $t=0$.

\subsubsection{Symmetric case}
The Weyl group $\mathring{W}$ acts linearly on  $\mathbb{F}[\mathring{P}]$ by $w(e^{\lambda})=e^{w(\lambda)}$ for each $w\in \mathring{W}$ and $\lambda\in \mathring{P}$.
\begin{defi}
We define an inner product on $\mathbb{F}[x_1^{\pm1},...,x_l^{\pm1}]$ by 
$$\langle f,g\rangle^{\prime}_{sym}:=\text{the constant term of}\; fg\De(x)\in \mathbb{F}.$$
\end{defi}

\begin{defi}
The set of symmetric Macdonald-Koornwinder polynomials $\{P_{\lambda}(x,q,t,a,b,c,d)\}_{\lambda\in \mathring{P}}$ is a collection of elements in $\mathbb{F}[\mathring{P}]^{\mathring{W}}$ indexed by $\mathring{P}_+$ with the following properties:
\item $\mathrm{(1)}$ $\langle P_{\lambda},P_{\mu}\rangle^{\prime}_{sym}=0$ if $\lambda \neq \mu$;
\item $\mathrm{(2)}$ $P_{\lambda}=X^{\lambda}+\underset{\mu\succ \lambda}{\sum}c_{\mu}X^{\mu}$.
\end{defi}
We set
$$\bar{P}_{\lambda}:=\underset{t\to 0}{\lim}P_{\lambda}.$$
Let $\langle-,-\rangle_{sym}$ be a specialization of $\langle-,-\rangle_{sym}^{\prime}$ at $t=0$.
We abbreviate $\bar{E}_{\lambda}(x_1,...,x_l,q)$, $E^{\dag}_{\lambda}(x_1,...,x_l,q)$, $\bar{P}_{\lambda}(x_1,...,x_l,q)$ and  $P^{\dag}_{\lambda}(x_1,...,x_l,q)$ as $\bar{E}_{\lambda}(X,q)$, $E^{\dag}_{\lambda}(X,q)$, $\bar{P}_{\lambda}(X,q)$ and  $P^{\dag}_{\lambda}(X,q),$ respectively.
\begin{prop}[\cite{Ion} Theorem 4.2]\label{symmetric=nonsymmetric}
For each $\lambda\in \mathring{P}_{+}$, we have $$\bar{P}_{\lambda}(X^{-1},q^{-1})=\bar{E}_{\lambda}(X^{-1},q^{-1}).$$
\end{prop}

\subsection{Representation of $\bo_{-}$ and $\Cg$ and their Euler-Poincar\'e-pairing}\label{representation}
\subsubsection{Representations of $\bo_-$}
For each $\bo_-$-module $M$ and $\lambda\in P$, we set $M_{\lambda}:=\{m\in M|\;hm=\lambda(h)m\;\mathrm{for}\;h\in \h\}.$ Let $\mathfrak{B}$ be the full subcategory of the category of $U(\bo_-)$-module such that a $\bo_-$-module $M$ is an object of $\mathfrak{B}$ if and only if $M$ has a weight decomposition $$M=\underset{\lambda\in P}{\bigoplus}M_{\lambda},$$ where $M_{\lambda}$ has at most countable dimension for all $\lambda\in P$. We set $\mathrm{wt}\;M:=\{\lambda\in P\mid M_{\lambda}\neq\{0\}\}.$ Let $\mathfrak{B^{\prime}}$ be the full subcategory of $\mathfrak{B}$ such that $M\in \mathfrak{B}$ is an object of $\mathfrak{B^{\prime}}$ if and only if $M$ is a $\bo_{-}$-module such that the set of weights $\mathrm{wt}\;M$ is contained in $\underset{i=1,...,k}{\bigcup}(\mu_i-Q_+)$ for some $\mu_i\in P$, and every weight space is finite dimensional. Let $\mathfrak{B}_0$ be the full subcategory of $\mathfrak{B^{\prime}}$ consisting of finite dimensional $\bo_{-}$-modules.
For each $M\in\mathfrak{B^{\prime}}$, we define a graded character of $M$ by the following formal sum
$$\mathrm{gch}\;M:=\underset{\lambda-m\delta \in \mathring{P}\oplus \frac{\mathbb{Z}}{2}\delta}{\sum}q^{m}X^{\lambda}\mathrm{dim}_{\mathbb{C}}\;\mathrm{Hom}_{\mathring{\h}\oplus \C d}\;(\C_{\lambda-m\delta},M),$$ where $\C_{\lambda-m\delta}$ is a 1-dimensional $\mathring{\h}\oplus \C\delta$-module with its weight $\lambda-m\delta.$
For each $\Lambda\in P$,  let $\mathbb{C}_{\Lambda}^{\prime}$ be the 1-dimensional $\mathfrak{h}$-module with its weight $\Lambda$, and $\mathbb{C}_{\Lambda}$ be the 1-dimensional simple module of $\bo_{-}$ with its weight $\Lambda.$ For each $\Lambda \in P$, we set $P(\Lambda):=U(\bo_{-})\underset{U(\mathfrak{h})}{\otimes}\mathbb{C}_{\Lambda}^{\prime}$ and $N(\Lambda):=\sum_{\mu\in P\backslash\{\Lambda\}}P(\Lambda)_{\mu}.$  Then $N(\Lambda)$ is a $\bo_-$-submodule of $P(\Lambda)$ and $\C_{\Lambda}\cong P(\Lambda)/N(\Lambda).$
\begin{prop}\label{projective cover of b}
For each $\Lambda\in P$, the $\bo_-$-module $P(\Lambda)=U(\bo_{-})\underset{U(\mathfrak{h})}{\otimes}\mathbb{C}_{\Lambda}^{\prime}$ is a projective cover of $\mathbb{C}_{\Lambda}$ in $\mathfrak{B}.$
\end{prop}
\begin{proof}
For each $M\in \mathfrak{B},$ we have $\mathrm{Hom}_{\mathfrak{B}}(P(\Lambda),M)=\mathrm{Hom}_{\h}(\mathbb{C}_{\Lambda},M).$ Hence, $P(\Lambda)$ is a projective cover of $\C_{\Lambda}$ in $\mathfrak{B}.$
\end{proof}
\begin{prop}[\cite{FKM} Lemma 5.2]\label{enough projective of b}
The category $\mathfrak{B}$ has enough projectives.
\end{prop}
\begin{defi}
Let $M$ be a $\bo_-$-module with $\h$-weight decomposition $M=\underset{\mu\in \h^*}{\bigoplus}M_{\mu}$. Then $M^{\vee}:=\underset{\mu \in \h^*}{\bigoplus}M^{*}$ is a $\bo_{-}$-module with a $\bo_{-}$-action defined by
$$Xf(v):=-f(Xv)\; \mathrm{for}\; X\in\bo_-,\;f\in M^{\vee}\; \mathrm{and}\; v\in M.$$
\end{defi}

\begin{defi}
For each $M\in \mathfrak{B}^{\prime}$ and $N\in \mathfrak{B}_0$, we define the Euler-Poincar\'e-pairing $\langle M, N\rangle_{\mathrm{Ext}}$ as a formal sum by
$$\langle M, N\rangle_{\mathrm{Ext}}:=\underset{p\in \mathbb{Z}_+,m\in \mathbb{Z}/2}{\sum}(-1)^{p}q^{m}\mathrm{dim}_{\mathbb{C}}\;\mathrm{Ext}_{\mathfrak{B}}^{p}(M\otimes_{\mathbb{C}}\mathbb{C}_{m\delta},N^{\vee}).$$ 
\end{defi}
\begin{prop}\label{a}For each $M\in \mathfrak{B}^{\prime}$ and $N\in \mathfrak{B}_0$, the following hold:
\item $\mathrm{(1)}$ The pairing $\langle M,N\rangle_{\mathrm{Ext}}$ is an element of $\mathbb{C}(\!(q^{1/2})\!)$;
\item $\mathrm{(2)}$ This pairing depends only on the graded characters of $M$ and $N.$
\end{prop}
\begin{proof}First, we prove (1). Let $S$ be the set of highest weight vectors of $M$. Since $\mathrm{wt}\;M$ is bounded from above, we have a surjection $\varphi^0:P^0:=\bigoplus_{v\in S}P(\mathrm{wt}(v))\to M,$ where $\mathrm{wt}(v)$ is the $\h$-weight of $v.$ Since $v\in S$ such that $(\mathrm{wt}(v)+Q_+\backslash \{0\})\cap \mathrm{wt}\;M=\emptyset$ is not an element of $\operatorname{Ker}\varphi^0$, the set $\mathrm{wt}\;\operatorname{Ker}\varphi^0$ is a proper subset of $\mathrm{wt}\;P^{0}$. For $\operatorname{Ker}\varphi^0,$ we define $\varphi^1:P^1\to \operatorname{Ker}\varphi^0$ in the same way. Repeating this procedure, we get a projective resolution $\cdots \to P^{1}\to P^{0} \to M\to 0$ such that $\mathrm{wt}\;P^{k+1}$ is a proper subset of $\mathrm{wt}\;P^{k}$ for all $k\in \mathbb{Z}_+.$ The complex $P^{\bullet}\otimes_{\mathbb{C}} \mathbb{C}_{m\delta}$ is a projective resolution of $M\otimes_{\mathbb{C}} \mathbb{C}_{m\delta}.$ For each $m\in\mathbb{Z}/2$, we have $\mathrm{wt}\;(P^{k}\otimes_{\mathbb{C}} \mathbb{C}_{m\delta})\cap \mathrm{wt}\;N=\emptyset$ for all $k\gg0$ since $N$ and every weight space of $M$ are finite dimensional. This implies $\mathrm{Ext}_{\mathfrak{B}}^{k}(M\otimes_{\mathbb{C}}\mathbb{C}_{m\delta},N^{\vee})=\{0\}$ for all $k\gg0.$ Hence $\sum_{k\in \mathbb{Z}_+}(-1)^kq^{m}\mathrm{dim}_{\mathbb{C}}\;\mathrm{Ext}_{\mathfrak{B}}^{k}(M\otimes_{\mathbb{C}}\mathbb{C}_{m\delta},N^{\vee})$ is well-defined. Since $\bo_-$-action on $P^{0}$ does not increase $d$-eigenvalues, and the set of weights of an object of $\mathfrak{B}^{\prime}$ is bounded from above, the intersection of the set of $d$-eigenvalues of $N^{\vee}$ and $P^{0}\otimes_{\mathbb{C}}\mathbb{C}_{m\delta}$ is empty for all $m\ll0.$ This implies the assertion. \\Next, we prove (2). Let $N^{\prime}$ be an object of $\mathfrak{B}_0$ such that $\mathrm{gch}\;N=\mathrm{gch}\;N^{\prime}.$ The sets of composition factors of $N$ and $N^{\prime}$ are the same. We denote the set of composition factors by $S.$ For each exact sequence $0\to L_1\to L_2\to L_3\to 0$, we have $\langle M,L_2\rangle_{\mathrm{Ext}}=\langle M,L_1\rangle_{\mathrm{Ext}}+\langle M,L_3\rangle_{\mathrm{Ext}}.$ This implies $\langle M,N\rangle_{\mathrm{Ext}}=\sum_{\mathbb{C}_{\Lambda}\in S}\langle M,\mathbb{C}_{\Lambda}\rangle_{\mathrm{Ext}}=\langle M,N^{\prime}\rangle_{\mathrm{Ext}}.$ Hence the assertion for the second argument follows. Let $K^0:=\bigoplus_{v\in S}N(\mathrm{wt}(v))$ be a $\bo_-$-submodule of $P^0$. We set $N^0:=M$ and $N^1:=\varphi^0(K^0)$. We define a $\bo_-$-submodule $N^2$ of $N^1$ in the same way for $N^1$ instead of $M$. Repeating this, we get a sequence of $\bo_-$-submodules $M=N^0 \supset N^1 \supset N^2 \supset \cdots.$ Since every weight space of $M$ is finite dimensional, for each $\mu\in P$, we have $N_{\mu}^{s}=\{0\}$ for $s\gg 0$ by construction. We can take a composition series $M=M^0\supset\cdots \supset M^s\supset M^{s+1}\supset \cdots $ of $M$ as a refinement of the above sequence of $\bo_-$-modules. Since $N$ is finite dimensional, for $s\gg 0$, we have $\mathrm{wt}(v)-\mathrm{wt}(u)\notin Q_{+}$ for each $v\in M^s$ and $u\in N.$ By taking a projective resolution of $M^s$ as in the proof of (1), we have $\mathrm{Ext}_{\mathfrak{B}}^{k}(M^s\otimes_{\mathbb{C}}\mathbb{C}_{m\delta},N^{\vee})=\{0\}$ for $s\gg 0.$ Using this composition series, we can prove the assertion for the first argument in the same way.\end{proof}Thanks to Proposition \ref{a}, we get a bilinear map from $\mathbb{C}(\!(q^{1/2})\!)[\mathring{P}]\times \mathbb{C}(\!(q^{1/2})\!)[\mathring{P}]$ to $\mathbb{C}(\!(q^{1/2})\!)[\mathring{P}]$, that we denote also $\langle-,-\rangle_{\mathrm{Ext}}$
\begin{prop}\label{extnonsym}
For each $M\in \mathfrak{B}^{\prime}$ and $N\in \mathfrak{B}_0$, we have $\langle\mathrm{gch}\;M,\mathrm{gch}\;N\rangle_{\mathrm{Ext}}=\langle\mathrm{gch}\;M,\mathrm{gch}\;N\rangle_{nonsym}$. 
\end{prop}
\begin{proof}
$\{\mathrm{gch}\;\C_{\Lambda}\}_{\Lambda\in P}$ and $\{\mathrm{gch}\;P(\Lambda)\}_{\Lambda\in P}$ are $\mathbb{C}(\!(q^{1/2})\!)$-basis of  $\mathbb{C}(\!(q^{1/2})\!)[\mathring{P}]$. Therefore, it suffices to check the assertion for $M=\C_{\Lambda}$ and $N=P(\Lambda)$. By the PBW theorem, we have $\mathrm{ch}\;P(\Lambda)=X^{\lambda}/\prod_{\al\in \De_-}(1-X^{\al})^{\mathrm{dim}_{\C}\;\g_{\al}}.$ Hence we have $\mathrm{ch}\;P(\Lambda)=X^{\Lambda}/\mathcal{C}|_{t=0}$. Hence we get
$$\langle\mathrm{gch}\;P(\Lambda),\;\mathrm{gch}\;\C_{\Lambda}\rangle_{\mathrm{Ext}}=1=\langle\mathrm{gch}\;P(\Lambda),\;\mathrm{gch}\;\C_{\Lambda}\rangle_{nonsym}.$$
The assertion follows.
\end{proof}
\subsubsection{Representations of $\Cg$}\label{representation of Cg}
Let $\Cg$-$\mathrm{mod}_{\mathrm{wt}}$ be the full subcategory of the category of $\Cg$-modules such that $M$ is an object of $\Cg$-$\mathrm{mod_{wt}}$ if and only if $M$ is a $\Cg$-module which has a weight decomposition $$M=\underset{\Lambda\in P}{\bigoplus}M_{\Lambda}$$ such that every weight space has at most countable dimension. Let $\Cg$-$\mathrm{mod^{int}}$ be the full subcategory of the category $\Cg$-$\mathrm{mod}_{\mathrm{wt}}$ such that an object $M$ of $\Cg$-$\mathrm{mod}_{\mathrm{wt}}$ is an object of $\Cg$-$\mathrm{mod^{int}}$ if and only if $M$ is an integrable $\mathring{\g}$-module and the set of weights $\mathrm{wt}\;M=\{\lambda\in P\mid M_{\lambda}\neq\{0\}\}$ is contained in $\underset{i=1,...,k}{\bigcup}(\mu_i-Q_+)$ for some $\mu_i\in P$ and every weight space is finite dimensional. For each $\lambda\in\mathring{P}_+$, $\mu\in\mathring{P}$ and $n,\;2m\in \mathbb{Z}$, we set $$P(\lambda+n\Lambda_0+m\delta)_{\mathrm{int}}:=U(\Cg)\underset{U(\mathring{\g}+\h)}{\otimes}V(\lambda+n\Lambda_0+m\delta)$$ and 
$$P(\mu+n\Lambda_0+m\delta)_{\mathrm{wt}}:=U(\Cg)\underset{U(\h)}{\otimes}\mathbb{C}_{\mu+n\Lambda_0+m\delta},$$
where $V(\lambda+n\Lambda_0+m\delta)$ is the highest weight simple module of $\mathring{\g}+\h$ with its highest weight $\lambda+n\Lambda_0+m\delta$ and $\mathbb{C}_{\mu+n\Lambda_0+m\delta}$ is the 1-dimensional module of $\h$ with its weight $\mu+n\Lambda_0+m\delta$. Let $\pi:\Cg\to\mathring{\g}$ be a homomorphism of Lie algebras defined by $$\pi|_{\mathring{\g}}=\mathrm{id}_{\mathring{\g}},\;\;\pi(\Cg_{\neq 0})=\{0\},$$ where $\Cg_{\neq 0}:=\{X\in \Cg\mid [d,X]\neq 0\}$. We can prove the following proposition in the same way as Proposition \ref{projective cover of b}, and we omit its proof.
\begin{prop}
For each $\mu\in\mathring{P}$ and $n$, $2m\in \mathbb{Z},$ the $\Cg$-module $P(\mu+n\Lambda_0+m\delta)_{\mathrm{wt}}$ is a projective module.
\end{prop}
\begin{prop}[\cite{CI} Proposition 2.3]
Let $\lambda\in\mathring{P}_+$ and $n,\;2m\in \mathbb{Z}$.
\item$\mathrm{(1)}$ $\pi^{*}V(\lambda+n\Lambda_0+m\delta)$ is a simple object in $\Cg\text{-}\mathrm{mod}^{\mathrm{int}}.$
\item$\mathrm{(2)}$ $P(\lambda+n\Lambda_0+m\delta)_{\mathrm{int}}$ is a projective cover of its unique simple quotient $\pi^{*}V(\lambda+n\Lambda_0+m\delta)$ in $\Cg\text{-}\mathrm{mod}^{\mathrm{int}}.$
\end{prop}
\begin{prop}
The categories $\Cg$-$\mathrm{mod}_{\mathrm{wt}}$ and $\Cg$-$\mathrm{mod}^{\mathrm{int}}$ have enough projectives.
\end{prop}
\begin{proof}
We can prove that $\Cg$-$\mathrm{mod}_{\mathrm{wt}}$ has enough projectives in the same way as Proposition $\ref{enough projective of b}$. Let $M$ be an object of $\Cg$-$\mathrm{mod}^{\mathrm{int}}.$ Since $M$ is an integrable $\mathring{\g}$-module, for each $\mathring{\g}$-highest weight vector $v\in M$ with its weight $\Lambda$, we have a morphism of $\Cg$-module $P(\Lambda)_{\mathrm{int}}\to M.$ Collecting them for all $\mathring{\g}$-highest weight vector, we obtain a surjection from a projective module to $M.$ 
\end{proof}
\begin{defi}\label{int pairing}
For each $M$, $N\in \Cg$-$\mathrm{mod^{int}}$ such that $N$ is finite dimensional, we define the Euler-Poincar\'e-pairing $\langle M, N\rangle_{int}$ as a formal sum by $$\langle M, N\rangle_{int}:=\underset{p\in \mathbb{Z}_+,m\in \mathbb{Z}/2}{\sum}(-1)^{p}q^{m}\mathrm{dim}_{\mathbb{C}}\;\mathrm{Ext}_{\Cg-\mathrm{mod^{int}}}^{p}(M\otimes_{\mathbb{C}}\mathbb{C}_{m\delta},N^{\vee}).$$
\end{defi}
We can prove the following proposition in the same way as Proposition $\ref{a}$, and we omit its proof.
\begin{prop}For each $M$, $N\in \Cg$-$\mathrm{mod^{int}}$ such that $N$ is finite dimensional, the following hold:
\item $\mathrm{(1)}$ The pairing $\langle M, N\rangle_{int}$ is an element of $\mathbb{C}(\!(q^{1/2})\!)$;
\item $\mathrm{(2)}$ This pairing depends only on the graded characters of $M$ and $N.$ 
\end{prop}
\section{Demazure modules}
We continue to work in the setting of the previous section.
\subsection{Representations of $\g$}
\subsubsection{Highest weight simple module}
\begin{defi}
Let $\Lambda\in P$ and let $\mathbb{C}_{\Lambda}$ be the corresponding $1$-dimensional module of $\bo_+$. The Verma module $M(\Lambda)$ of highest weight $\Lambda$ is a $\g$-module defined by $$M(\Lambda):=U(\g)\underset{U(\bo_+)}{\otimes}\mathbb{C}_{\Lambda}.$$
\end{defi}
The Verma module $M(\Lambda)$ has a unique simple quotient (see \cite{Kac} Proposition 9.2). We denote it by $L(\Lambda)$.

\begin{thm}[see \cite{Kac} Proposition 3.7, Lemma 10.1 and \S 9.2]\label{Kac}For each $\Lambda \in P$, the following hold. 
\item $\mathrm{(1)}$ $L(\Lambda)$ is an integrable $\g$-module if and only if $\Lambda \in P_+$;
\item $\mathrm{(2)}$ For each $\Lambda \in P_+$ and $w\in W$, we have $\mathrm{dim}_{\C}\;L(\Lambda)_{w\Lambda}=1$;
\item $\mathrm{(3)}$ $L(\Lambda)$ has a $\h$-weight decomposition $$L(\Lambda)=\underset{\mu \in P}{\bigoplus}L(\Lambda)_{\mu}$$ and $L(\Lambda)_{\mu}$ is finite-dimensional for all $\mu \in P$.
\end{thm} We remark that $\mathrm{gch}\;L(\Lambda)$ is well-defined thanks to Theorem \ref{Kac} (3).
\subsubsection{Realization of $L(\Lambda_0)$}\label{realization}

\begin{defi}[Heisenberg algebra]
For each $l\in \mathbb{N}$, let $S_l$ be a unital $\mathbb{C}$-algebra generated by $x_{i,n}$ $(i=1,...,l,\; 0\neq n\in \mathbb{Z})$ and $K$ which satisfy the following conditions:
\item $\mathrm{(1)}$ $[x_{i,n}, x_{j,m}]=n\delta_{i,j}\delta_{n,-m}K$;
\item $\mathrm{(2)}$ $[K,S_l]=0$.
\end{defi}
We set $R=\mathbb{C}[y_{i,n}\mid i\in \{1,...,l\},\;n\in\mathbb{N}]$. We define a representation $p:S_{l}\to \mathrm{End}_{\mathbb{C}}\:(R)$ by
$$p(x_{i,-n})=y_{i,n},\;\;p(x_{i,n})=n\frac{\partial}{\partial y_{i,n}},\;\;p(K)=\mathrm{id}_{R}\;\;(n>0).$$
Let $\g_{im}:=\underset{n\in\mathbb{Z}\backslash\{0\}}{\bigoplus}\g_{n\delta}.$ The algebra $S_l$ is a $\mathbb{Z}$-graded algebra by setting $\mathrm{deg}\;x_{i,n}=n$ and $\mathrm{deg}\;K=0$, and $U(\g_{im}\oplus \C K)$ is a $\mathbb{Z}$-graded algebra by the $\mathbb{Z}$-grading induced from the adjoint action of the scaling element $d.$ For $\g$ of type $A_{2l}^{(2)},$ we have $\mathrm{dim}_{\C}\;\g_{n\delta}=l$ for $n\in \Z$, and we have the following.
\begin{prop}[see \cite{Kac} Proposition 8.4]\label{heisenberg}
The algebras $U(\g_{im}\oplus \C K)$ and $S_{l}$ are isomorphic as $\mathbb{Z}$-graded algebras.
\end{prop}
By Proposition $\ref{heisenberg}$, we identify $S_l$ with $U(\g_{im}\oplus \C K)$. Since $\mathring{\h}$ and $U(\g_{im}\oplus \C K)$ are mutually commutative, the following $\mathbb{C}$-algebra homomorphism $p_{\lambda}:U(\g_{im}\oplus \mathring{\h}\oplus \C K)\to \mathrm{End}_{\mathbb{C}}\;(R)\;(\lambda\in \mathring{P})$ is well-defined
$$p_{\lambda}|_{S_l}=p\;\; \mathrm{and}\;\; p_{\lambda}(h)=\lambda(h)\mathrm{id}_{R}\;\; \mathrm{for}\;\; h\in\mathring{\h}.$$
We denote this $U(\g_{im}\oplus \mathring{\h}\oplus \C K)$-module by $R_{\lambda}$.

\begin{thm}[\cite{LNX} Theorem 6.4]\label{frenkelkac}
We put $\tilde{p}:=\underset{\lambda \in \mathring{P}}{\prod}p_{\lambda}:U(\g_{im}\oplus \mathring{\h}\oplus \C K)\to \mathrm{End}_{\mathbb{C}}\;(\underset{\lambda \in \mathring{P}}{\bigoplus}R_{\lambda})$. Then $\tilde{p}$ extends to an algebra homomorphism $U(\g)\to \mathrm{End}_{\C}\;(\underset{\lambda \in \mathring{P}}{\bigoplus}R_{\lambda})$ and $\underset{\lambda \in \mathring{P}}{\bigoplus}R_{\lambda}$ is isomorphic to $L(\Lambda_0)$ as a $\g$-module.   
\end{thm}

\subsection{Thin and thick Demazure modules}
\begin{defi}
For each $w\in W$ and $\Lambda\in P_+$, we define $\bo_{-}$-modules 
$$D_{w\Lambda}:=U(\bo_-)v_{w\Lambda}^{*}\subset L(\Lambda)^{\vee}\; and\; D^{w\Lambda}:=U(\bo_-)v_{w\Lambda}\subset L(\Lambda).$$
Here $v_{w\Lambda}\in L(\lambda)_{w\Lambda}$ and $v_{w\Lambda}^{*}\in (L(\Lambda)_{w\Lambda})^{*}$ are nonzero vectors. By Theorem \ref{Kac} $\mathrm{(3)}$, these vectors are unique up to scalars. Hence $D_{w\Lambda}$ and $D^{w\Lambda}$ are well-defined. We call $D_{w\Lambda}$ a thin Demazure module and  $D^{w\Lambda}$ a thick Demazure module.

\end{defi}
\begin{remark}
A Demazure module in $\cite{Kum}$ means the thin Demazure module $D_{w\Lambda}$.
\end{remark}
\begin{lem}[\cite{Kac} Proposition 3.6]\label{extremal vector}
For each $w\in W$, $\Lambda\in P_+$ and $\al\in \De_{+}$, we have
\begin{equation*}
v_{s_{\al}w\Lambda}\in
\begin{cases}
\g_{-\al}^{\langle w\Lambda, \check{\al}\rangle}v_{w\Lambda}&(\langle w\Lambda ,\check{\al}\rangle>0)\\
\g_{\al}^{-\langle w\Lambda, \check{\al}\rangle}v_{w\Lambda}&(\langle w\Lambda ,\check{\al}\rangle<0)\\
\mathbb{C}v_{w\Lambda}&(\langle w\Lambda ,\check{\al}\rangle=0)
\end{cases},
\end{equation*}
where $\g_{\al}^{m}= \{X_1X_2\cdots X_m\in U(\g)\mid X_i\in\g_{\al}\}.$
\end{lem}
Lemma \ref{demazure} and Corollary \ref{quotient of demazure} in the below are proved in \cite{CK} for the dual of the untwisted affine Lie algebra. The proofs in \cite{CK} are also valid for type $A_{2l}^{(2)}$.
\begin{lem}[\cite{CK} Corollary 4.2]\label{demazure}
For each $w,v\in W$ and $\Lambda\in P_+$, we have the following:
\item $\mathrm{(1)}$ If $w\leq v$, then $D^{v\Lambda}\subseteq D^{w\Lambda}$;
\item $\mathrm{(2)}$ If $w$ and $v$ are minimal representatives of cosets in $W/W_{\Lambda}$ and $D^{v\Lambda}\subseteq D^{w\Lambda}$, then $w\leq v.$
\end{lem}
Lemma $\ref{demazure}$ allows us to define as follows:
\begin{defi}
For each $w\in W$ and $\Lambda\in P_+$, we define a $U(\bo_-)$-module $\mathbb{D}^{w\Lambda}$ as $$\D^{w\Lambda}:=D^{w\Lambda}/\underset{w<v}{\sum}D^{v\Lambda}.$$ We call this module a Demazure slice.
\end{defi}

\begin{prop}[\cite{Kat} Corollary 2.22]
For each $\Lambda \in P_+$ and $S\subset W$, there exists $S^{'}\subset W$ such that 
$$\underset{w\in S}{\bigcap}D^{w\lambda}=\underset{w\in S^{'}}{\sum}D^{w\Lambda}.$$
\end{prop}

\begin{cor}[\cite{CK} Corollary 4.4]\label{quotient of demazure}
For each $w,\;v\in W$ and $\Lambda\in P_{+}$, we have
\begin{equation*}
(D^{w\Lambda}\cap D^{v\Lambda})/(D^{v\Lambda}\cap\underset{u>w}{\sum}D^{u\Lambda})=\D^{w\Lambda}\;or\; \{0\}.
\end{equation*}
\end{cor}

\subsection{Level one case}
In this subsection, we consider level one Demazure modules. The unique level one dominant integral weight of $A_{2l}^{(2)}$ is $\Lambda_{0}$. From \eqref{weyl group},
$$\mathring{P}\ni \lambda \mapsto \lambda +\Lambda_{0}+\frac{(\lambda|\lambda)}{2}\delta\in W\Lambda_{0}$$ is a bijection. For each $\lambda\in \mathring{P}$, we set
$$D_{\lambda}:=D_{\pi_{\lambda}},\;D^{\lambda}:=D^{\pi_{\lambda}},\;\D^{\lambda}:=\D^{\pi_{\lambda}}.$$ 

\begin{lem}
For each $\lambda,\;\mu\in \mathring{P}$, we have $D^{\lambda}\subsetneq D^{\mu}$ if and only if
 $\mu \succ \lambda.$
\end{lem}
\begin{proof}
If $D^{\lambda}\subsetneq D^{\mu}$, then we have $\pi_{\mu}< \pi_{\lambda}$ by Lemma \ref{demazure}. Then, Lemma \ref{cherednik order and bruhat order} (1) implies $\mu\succ \lambda$. Conversely, we assume that $\mu\succ \lambda$. There exists $w\in W$ such that $\mu\succ \lambda=w(\!(\mu)\!)$. Let $w=s_{i_1}\cdots s_{i_n}$ be a reduced expression of $w$ such that $(s_{i_{k+1}}\cdots s_1)(\!(\mu)\!)\succ (s_{i_k}s_{i_{k+1}}\cdots s_1)(\!(\mu)\!)$ for all $k$. If $n=1$, then Lemma \ref{cherednik order and bruhat order} (2) implies $\pi_{\mu}<\pi_{\lambda}$. Hence, we have $D^{\lambda}\subsetneq D^{\mu}$. If $n>1$, then we have $D^{\lambda}\subsetneq D^{(s_{i_2}\cdots s_{i_n})(\!(\mu)\!)}\subsetneq \cdots \subsetneq D^{\mu}$ inductively. 
\end{proof}
\begin{thm}[\cite{Ion} Theorem 1]\label{characterequality}
For each $\lambda\in \mathring{P}$, we have
$$\mathrm{gch}\;D_{\lambda}=q^{\frac{(b|b)}{2}}\bar{E}_{\lambda}(X^{-1},q^{-1}).$$
\end{thm}

\subsection{Weyl modules}

\begin{defi}[\cite{CIK} \S3.3]\label{global weyl}
For each $\lambda\in \mathring{P}_+$, the global Weyl module is a cyclic $\Cg$-module $W(\lambda)$ generated by a vector $v_{\lambda}$ that satisfies following relations:
\item $\mathrm{(1)}$ $hv_{\lambda}=\lambda(h)v_{\lambda}$ for each $h\in \h$;
\item $\mathrm{(2)}$ $e_{-\alpha}^{\langle\lambda,\check{\al}\rangle+1}v_{\lambda}=0$ for each $\al\in\mathring{\De}_{+}$;
\item $\mathrm{(3)}$ $\mathfrak{Cn}_+v_{\lambda}=0$.
\end{defi}

\begin{defi}[\cite{CIK} \S3.5 and \S7.2]For each $\lambda\in \mathring{P}$, the local Weyl module is a cyclic $\Cg$-module $W(\lambda)_{loc}$ generated by a vector $v_{\lambda}$ satisfies relations $\mathrm{(1)}$, $\mathrm{(2)}$, $\mathrm{(3)}$ of Definition \ref{global weyl} and 
\item $\mathrm{(4)}$ $Xv_{\lambda}=0$ for $X\in \Cg_{im}.$
\end{defi}
\begin{thm}[\cite{CIK} Theorem 2]\label{demazure=weyl}
Let $\lambda \in \mathring{P}_+$. Then $D_{\lambda}\otimes_{\C} \mathbb{C}_{(\lambda\mid\lambda)\delta/2-\Lambda_0}$ is isomorphic to $W(\lambda)_{loc}$ as $\Cg$-module, where $\mathbb{C}_{(\lambda\mid\lambda)\delta/2-\Lambda_0}$ is the 1-dimensional module with its $\h$-weight $(\lambda|\lambda)\delta/2-\Lambda_0.$
\end{thm}
\begin{cor}\label{characterequality2}
For each $\lambda\in \mathring{P}_+,$ we have $$\mathrm{gch}\;W(\lambda)_{loc}=q^{\frac{(b|b)}{2}}\bar{P}_{\lambda}(X^{-1},q^{-1}).$$
\end{cor}
\begin{proof}
By Theorem \ref{demazure=weyl}, we have $$\mathrm{gch}\;W(\lambda)_{loc}=\mathrm{gch}\;D_{\lambda}.$$ By Proposition \ref{symmetric=nonsymmetric} and Theorem \ref{characterequality}, the assertion follows.
\end{proof}
\begin{thm}[\cite{CI} Theorem 2.5 (3), Theorem 4.7 and \cite{Kle} Theorem 7.21]\label{Ext}
For each $\lambda$, $\mu\in \mathring{P}_+$ and $m\in \mathbb{Z}/2$, we have 
$$\mathrm{dim}_{\mathbb{C}}\;\mathrm{Ext}^{n}_{\Cg\text{-}\mathrm{mod^{int}}}(W(\lambda)\otimes_{\C}\mathbb{C}_{m\delta},W(\mu)_{loc}^{\vee})=\delta_{m,0}\delta_{0,n}\delta_{\lambda,\mu}.
$$
\end{thm}

\begin{cor}\label{orthogonality}
For each $\lambda$, $\mu\in \mathring{P}_+$, we have $\langle \mathrm{gch}\;W(\lambda),\mathrm{gch}\;W(\mu)_{loc}\rangle_{int}=\delta_{\lambda,\mu}$.
\end{cor}
\begin{proof}
The assertion follows from Definition \ref{int pairing} and Theorem \ref{Ext}.
\end{proof}
\subsubsection{Extensions between Weyl modules in $\mathfrak{B}$}\label{extensions between weyl modules in b}
In this subsection, we prove the following corollary of Theorem \ref{Ext}.
\begin{thm}\label{ext}
For each $\lambda$, $\mu\in \mathring{P}_+,$ $m\in \mathbb{Z}/2$ and $n\in \Z_+$, we have 
$$\mathrm{dim}_{\mathbb{C}}\;\mathrm{Ext}^{n}_{\mathfrak{B}}(W(\lambda)\otimes_{\C}\mathbb{C}_{m\delta},W(\mu)_{loc}^{\vee})=\delta_{m,0}\delta_{0,n}\delta_{\lambda,\mu}.$$
\end{thm}

\begin{defi}[\cite{Gro} \S2.1]
Let $\mathfrak{C}$, $\mathfrak{D}$ be abelian categories. A contravariant $\delta$-functor from $\mathfrak{C}$ to $\mathfrak{D}$ consists of the following data:
\item $\mathrm{(a)}$ A collection $T=\{T^{i}\}_{i\in\mathbb{Z}_+}$ of contravariant additive functors from $\mathfrak{C}$ to $\mathfrak{D}$;
\item $\mathrm{(b)}$ For each exact sequence $0\to M^{\prime\prime}\to M\to M^{\prime}\to 0$, a collection of morphisms $\{\delta^{n}:T^n(M^{\prime\prime})\to T^{n+1}(M^{\prime})\}_{n\in\Z_+}$ with the following conditions:
\item$\mathrm{(1)}$ For each exact sequence $0\to M^{\prime\prime}\to M\to M^{\prime}\to 0$, there is a long exact sequence 
\begin{eqnarray*}
0&\to& T^0(M^{\prime})\to T^0(M)\to T^0(M^{\prime\prime})\overset{\delta^0}{\to}\\ &\to& T^1(M^{\prime})\to\; \cdots\;\to T^{n-1}(M^{\prime\prime})\overset{\delta^{n-1}}{\to}\\  &\to& T^n(M^{\prime})\to T^{n}(M)\to T^{n}(M^{\prime\prime})\overset{\delta^{n}}{\to}\cdots;
\end{eqnarray*}
\item$\mathrm{(2)}$ For each morphism of short exact sequence\[
  \begin{CD}
     0 @>>>  M^{\prime\prime}  @>>>  M  @>>>  M^{\prime}  @>>>  0 \\
    @.     @VVV  @VVV  @VVV   @. \\
     0 @>>>  N^{\prime\prime} @>>>  N @>>>  N^{\prime} @>>>  0
  \end{CD},
\] we have the following commutative diagram \[
  \begin{CD}
     T^{n-1}(N^{\prime\prime}) @>{\delta^{n-1}}>> T^{n}(N^{\prime}) \\
  @VVV    @VVV \\
     T^{n-1}(M^{\prime\prime})   @>{\delta^{n-1}}>>  T^{n}(M^{\prime})
  \end{CD}.
\]
\end{defi}
\begin{defi}[\cite{Gro} \S2.1]
For each contravariant $\delta$-functors $T=\{T^i\}_{i\in\Z_+}$ and $S=\{S^i\}_{i\in\Z_+}$, a morphism of $\delta$-functor from $T=\{T^i\}_{i\in\Z_+}$ to $S=\{S^i\}_{i\in\Z_+}$ is a collection of natural transformations $F=\{F^i: T^i\to S^i\}_{i\in\Z_+}$ with the following condition:
\item $\mathrm{(*)}$ For each exact sequence $0\to M^{\prime\prime}\to M\to M^{\prime}\to 0$, the following diagram is commutative  \[
  \begin{CD}
     T^{n-1}(M^{\prime\prime}) @>{\delta^{n-1}}>> T^{n}(M^{\prime}) \\
  @V{F^{n-1}(M^{\prime\prime})}VV    @V{F^n(M^{\prime})}VV \\
     S^{n-1}(M^{\prime\prime})   @>{\delta^{n-1}}>>  S^{n}(M^{\prime})
  \end{CD}.
\]
\end{defi}
\begin{defi}[\cite{Gro} \S2.2]
A contravariant $\delta$-functor $T=\{T^i\}_{i\in \mathbb{Z}_+}$is called a universal $\delta$-functor if for each $\delta$-functor $S=\{S^i\}_{i\in \mathbb{Z}_+}$ and for each natural transformation $F^{0}:T^0\to S^0$, there exists a unique morphism of $\delta$-functor $\{F^{i}:T^i\to S^i\}_{i\in \mathbb{Z}_+}.$
\end{defi}

\begin{defi}[\cite{Gro} \S2.2]
An additive functor  $F:\mathfrak{C}\to\mathfrak{D}$ is called coeffaceable if for each object $M$ of $\mathfrak{C}$, there is a epimorphism $P\to M$ such that $F(P)=0$. 
\end{defi}

\begin{thm}[\cite{Gro} Proposition 2.2.1]\label{grothendieck}
For each $\mathfrak{C}$, $\mathfrak{D}$ be abelian categories and let $T=\{T^i\}_{i\in\Z_+}$ be a contravariant $\delta$-functor from $\mathfrak{C}$ to $\mathfrak{D}$. If $T^{i}$ is coeffaceable for $i>0$, then $T$ is universal.
\end{thm}
\begin{lem}[Shapiro's lemma]\label{Shapiro's lemma}
For each $M\in \mathfrak{B}$, $N\in \Cg\text{-}\mathrm{mod}_{\mathrm{wt}}$ and $n\in \mathbb{Z}_+$, we have 
$$\mathrm{Ext}^{n}_{\mathfrak{B}}(M,N)=\mathrm{Ext}^{n}_{\Cg\text{-}\mathrm{mod}_{\mathrm{wt}}}(U(\Cg)\underset{U(\mathfrak{b}_-)}{\otimes} M,N).$$
\end{lem}
\begin{proof}
Let $P^{\bullet}\to M\to 0$ be a projective resolution of $M$ in $\mathfrak{B}$. Since $U(\Cg)$ is free over $U(\bo_-)$, the complex $U(\Cg)\underset{U(\bo_-)}{\otimes} P^{\bullet}$ is a projective resolution of $U(\Cg)\underset{U(\mathfrak{b}_-)}{\otimes} M$ in $\Cg$-$\mathrm{mod}_{\mathrm{wt}}.$ The assertion follows from the Frobenius reciprocity.
\end{proof}
\begin{lem}\label{extensions}We have the following: 
\item$\mathrm{(1)}$ For each $M,\;N\in \Cg\text{-}\mathrm{mod}^{\mathrm{int}}$, we have $$\mathrm{Ext}^{k}_{\Cg\text{-}\mathrm{mod}_{\mathrm{wt}}}(M,N^{\vee})=\mathrm{Ext}^{k}_{\Cg\text{-}\mathrm{mod}^{\mathrm{int}}}(M,N^{\vee})\;\;\;k\in \mathbb{Z}_+;$$

\item$\mathrm{(2)}$ For each $N\in \Cg\text{-}\mathrm{mod^{int}}$, we have $$\mathrm{Ext}^{k}_{\Cg\text{-}\mathrm{mod}_{\mathrm{wt}}}(U(\Cg)\underset{U(\mathfrak{b}_-)}{\otimes} \mathbb{C}_0,N^{\vee})=\mathrm{Ext}^{k}_{\Cg\text{-}\mathrm{mod}_{\mathrm{wt}}}(\mathbb{C}_0,N^{\vee})\;\;\;k\in \mathbb{Z}_+.$$
\end{lem}
\begin{proof}
First, we prove the first assertion. The sets of functors $\{\mathrm{Ext}^{k}_{\Cg\text{-}\mathrm{mod}_{\mathrm{wt}}}(-,N^{\vee})\}_{k\in\Z_+}$ and $\{\mathrm{Ext}^{k}_{\Cg\text{-}\mathrm{mod}^{\mathrm{int}}}(-,N^{\vee})\}_{k\in\Z_+}$ are contravariant $\delta$-functors from $\Cg\text{-}\mathrm{mod}^{\mathrm{int}}$ to the category of vector spaces. From Theorem \ref{grothendieck},  $\{\mathrm{Ext}^{k}_{\Cg\text{-}\mathrm{mod}^{\mathrm{int}}}(-,N^{\vee})\}_{k\in\Z_+}$ is a universal $\delta$-functor. We prove $\{\mathrm{Ext}^{k}_{\Cg\text{-}\mathrm{mod}_{\mathrm{wt}}}(-,N^{\vee})\}_{k\in\Z_+}$ is also a universal $\delta$-functor. From Theorem \ref{grothendieck}, it is sufficient to show $\mathrm{Ext}^{l}_{\Cg\text{-}\mathrm{mod}_{\mathrm{wt}}}(P(\lambda+n\Lambda_0+m\delta)_{\mathrm{int}},N^{\vee})=\{0\}$ for $\lambda\in \mathring{P}_+$, $n,$ $2m\in\mathbb{Z}$ and $l>0$. From the BGG-resolution, we have an exact sequence 
$$\cdots\to\underset{w\in \mathring{W},l(w)=n+1}{\bigoplus}\mathring{M}(w\circ\lambda+n\Lambda_0+m\delta)\to\underset{w\in \mathring{W},l(w)=n}{\bigoplus}\mathring{M}(w\circ\lambda+n\Lambda_0+m\delta)\to\cdots$$
$$\cdots\to\underset{w\in \mathring{W},l(w)=1}{\bigoplus}\mathring{M}(w\circ\lambda+n\Lambda_0+m\delta)\to \mathring{M}(\lambda+n\Lambda_0+m\delta)\to V(\lambda+n\Lambda_0+m\delta)\to0,$$
where$$\mathring{M}(\mu):=U(\mathring{\g}+\h)\underset{U(\mathring{\bo}_{+}+\h)}{\otimes}\mathbb{C}_{\mu}.$$ Since $U(\Cg)$ is free over $U(\mathring{\g}+\h)$, by tensoring $U(\Cg)$, we obtain a projective resolution $P^{\bullet}\to P(\lambda+n\Lambda_0+m\delta)_{\mathrm{int}}\to 0$ of $P(\lambda+n\Lambda_0+m\delta)_{\mathrm{int}}$ in $\Cg$-$\mathrm{mod}_{\mathrm{wt}}$ such that $P^{n}=\underset{w\in \mathring{W},l(w)=n}{\bigoplus}U(\Cg)\underset{U(\mathring{\g}+\h)}{\otimes}\mathring{M}(w\circ\lambda+n\Lambda_0+m\delta).$ For each $l(w)>0,$ since $U(\Cg)\underset{U(\mathring{\g}+\h)}{\otimes}\mathring{M}(w\circ\lambda+n\Lambda_0+m\delta)$ does not have a $\mathring{\g}$-integrable quotient, we have $\mathrm{Hom}_{\Cg\text{-}\mathrm{mod}_{\mathrm{wt}}}(U(\Cg)\underset{U(\mathring{\g}+\h)}{\otimes}\mathring{M}(w\circ\lambda+n\Lambda_0+m\delta),N^{\vee})=\{0\}.$ This implies $\mathrm{Ext}^{l}_{\Cg\text{-}\mathrm{mod}_{\mathrm{wt}}}(P(\lambda+n\Lambda_0+m\delta)_{\mathrm{int}},N^{\vee})=\{0\}$ for $l>0$. Hence $\{\mathrm{Ext}^{k}_{\Cg\text{-}\mathrm{mod}_{\mathrm{wt}}}(-,N^{\vee})\}_{k\in \Z_+}$ is a universal $\delta$-functor by Theorem \ref{grothendieck}. Since $\mathrm{Ext}^{0}_{\Cg\text{-}\mathrm{mod}_{\mathrm{wt}}}(-,N^{\vee})=\mathrm{Ext}^{0}_{\Cg\text{-}\mathrm{mod}_{\mathrm{int}}}(-,N^{\vee})$, the assertion follows.\\
Next, we prove the second assertion. Two sets of functors $\{\mathrm{Ext}^{k}_{\Cg\text{-}\mathrm{mod}_{\mathrm{wt}}}(U(\Cg)\underset{U(\mathfrak{b}_-)}{\otimes} \mathbb{C}_0,(-)^{\vee})\}_{k\in\Z_+}$ and $\{\mathrm{Ext}^{k}_{\Cg\text{-}\mathrm{mod}_{\mathrm{wt}}}(\mathbb{C}_0,(-)^{\vee})\}_{k\in\Z_+}$ are contravariant $\delta$-functors from $\Cg\text{-}\mathrm{mod}^{\mathrm{int}}$ to the category of vector spaces. Since $\C_0$ is an object of $\Cg$-$\mathrm{mod}^{\mathrm{int}},$ we can prove that the latter is a universal $\delta$-functor by the same argument as in the proof of (1). We show that $\mathrm{Ext}^{l}_{\Cg\text{-}\mathrm{mod}_{\mathrm{wt}}}(U(\Cg)\underset{U(\mathfrak{b}_-)}{\otimes} \mathbb{C}_0,P(\lambda+n\Lambda_0+m\delta)^{\vee}_{\mathrm{int}})=\{0\}$ for each $l>0.$ For each $w\in\mathring{W}$, by the PBW theorem and the Frobenius reciprocity, we have 
\begin{eqnarray*}
&&\mathrm{Hom}_{\Cg}(U(\Cg)\underset{U(\mathfrak{b}_-)}{\otimes} \mathbb{C}_0,(U(\Cg)\otimes_{U(\mathring{\g}+\h)}\mathring{M}(w\circ\lambda+n\Lambda_0+m\delta))^{\vee})\\&=&\mathrm{Hom}_{\Cg}(U(\Cg)\otimes_{U(\mathring{\g}+\h)}\mathring{M}(w\circ\lambda+n\Lambda_0+m\delta),(U(\Cg)\underset{U(\mathfrak{b}_-)}{\otimes}\mathbb{C}_0)^{\vee})\\
&=&\mathrm{Hom}_{\mathring{\bo}_++\h}(\C_{w\circ\lambda+n\Lambda_0+m\delta},(U(\Cg)\underset{U(\mathfrak{b}_-)}{\otimes}\mathbb{C}_0)^{\vee})\\
&=&\mathrm{Hom}_{\mathring{\bo}_++\h}(\C_{w\circ\lambda+n\Lambda_0+m\delta},(U(\mathring{\bo}_++\h)\underset{U(\h)}{\otimes}\mathbb{C}_0)^{\vee})\\
&=&\mathrm{Hom}_{\mathring{\bo}_++\h}(U(\mathring{\bo}_++\h)\underset{U(\h)}{\otimes}\mathbb{C}_0,\C_{-w\circ\lambda-n\Lambda_0-m\delta})\\
&=&\mathrm{Hom}_{\h}(\mathbb{C}_0,\C_{-w\circ\lambda-n\Lambda_0-m\delta}).
\end{eqnarray*}
If $l(w)>0$, then $\mathrm{Hom}_{\mathring{\bo}_++\h}(\C_{w\circ\lambda+n\Lambda_0+m\delta},(U(\mathring{\bo}_++\h)\underset{\C}{\otimes}\mathbb{C}_0)^{\vee})=\{0\}$. Using the projective resolution of $P(\lambda+n\Lambda_0+m\delta)_{\mathrm{int}}$ considered in the proof of (1), this implies $\mathrm{Ext}^{l}_{\Cg\text{-}\mathrm{mod}_{\mathrm{wt}}}(U(\Cg)\underset{U(\mathfrak{b}_-)}{\otimes} \mathbb{C}_0,P(\lambda+n\Lambda_0+m\delta)^{\vee}_{\mathrm{int}})=\{0\}$ for each $l>0$. Hence $\mathrm{Ext}^{k}_{\Cg\text{-}\mathrm{mod}_{\mathrm{wt}}}(U(\Cg)\underset{U(\mathfrak{b}_-)}{\otimes} \mathbb{C}_0,(-)^{\vee})=\{0\}$ is a universal $\delta$-functor. Since $\mathrm{Ext}^{0}_{\Cg\text{-}\mathrm{mod}_{\mathrm{wt}}}(U(\Cg)\underset{U(\mathfrak{b}_-)}{\otimes} \mathbb{C}_0,N^{\vee})=\mathrm{Ext}^{0}_{\Cg\text{-}\mathrm{mod}_{\mathrm{wt}}}(\mathbb{C}_0,N^{\vee})$, the assertion follows.
\end{proof}
\begin{lem}\label{extension dual}For each $M,$ $N\in \mathfrak{B}$, we have $$\mathrm{Ext}_{\mathfrak{B}}^n(M,N^{\vee})=\mathrm{Ext}_{\mathfrak{B}}^n(\mathbb{C}_{0},M^{\vee}\otimes_{\C} N^{\vee})\;\;\mathrm{for}\;\;n\in \mathbb{Z}_+.$$
\end{lem}
\begin{proof}
We show that $\{\mathrm{Ext}_{\mathfrak{B}}^{n}(\mathbb{C}_{0},(-)^{\vee}\otimes_{\mathbb{C}}N^{\vee})\}_{n\in \mathbb{Z}_+}$ is a universal $\delta$-functor. For each injective object $I\in \mathfrak{B},$ the object $I\otimes_{\C} N^{\vee}$ is an injective object in $\mathfrak{B}$. Hence we have $\mathrm{Ext}_{\mathfrak{B}}^{k}(\mathbb{C}_{0},P^{\vee}\otimes_{\mathbb{C}}N^{\vee})=\{0\}$ for each projective object $P\in \mathfrak{B}$ and $k\in \mathbb{N}.$ From Theorem \ref{grothendieck}, this implies $\{\mathrm{Ext}_{\mathfrak{B}}^{n}(\mathbb{C}_{0},(-)^{\vee}\otimes_{\mathbb{C}}N^{\vee})\}_{n\in \mathbb{Z}_+}$ is a universal $\delta$-functor. For each $R\in \mathfrak{B}$, we have $\mathrm{Hom}_{\mathfrak{B}}(R,N^{\vee})=\mathrm{Hom}_{\mathfrak{B}}(\mathbb{C}_0,R^{\vee}\otimes_{\mathbb{C}} N^{\vee}).$ Since $\{\mathrm{Ext}_{\mathfrak{B}}^n(-,N^{\vee})\}_{n\in \mathbb{Z}_+}$ is a universal $\delta$-functor, this implies $\{\mathrm{Ext}_{\mathfrak{B}}^n(-,N^{\vee})\}_{n\in \mathbb{Z}_+}\cong \{\mathrm{Ext}_{\mathfrak{B}}^{n}(\mathbb{C}_{0},(-)^{\vee}\otimes_{\mathbb{C}}N^{\vee})\}_{n\in \mathbb{Z}_+}.$ Hence the assertion follows.
\end{proof}
\begin{remark}\label{extension dual2}
The conclusion of Lemma \ref{extension dual} remains valid if we replace $\mathrm{Ext}^{n}_{\mathfrak{B}}$ with $\mathrm{Ext}^{n}_{\Cg\text{-}\mathrm{mod}^{\mathrm{int}}}$ by the same argument. 
\end{remark}
\begin{thm}\label{compare extension}
For $M,$ $N\in \Cg\text{-}\mathrm{mod}^{\mathrm{int}},$ we have $$\mathrm{Ext}^{n}_{\mathfrak{B}}(M,N^{\vee})=\mathrm{Ext}^{n}_{\Cg\text{-}\mathrm{mod}^{\mathrm{int}}}(M,N^{\vee}).$$
\end{thm}
\begin{proof}
We have \begin{align*}
\mathrm{Ext}^{n}_{\mathfrak{B}}(M,N^{\vee})&=\mathrm{Ext}^{n}_{\mathfrak{B}}(\mathbb{C}_0,M^{\vee}\otimes_{\C}N^{\vee})\;\;\;\mathrm{from}\; \mathrm{Lemma}\; \ref{extension dual}\\
&=\mathrm{Ext}^{n}_{\Cg\text{-}\mathrm{mod}_{\mathrm{wt}}}(U(\Cg)\underset{U(\mathfrak{b}_-)}{\otimes} \mathbb{C}_0,M^{\vee}\otimes_{\C} N^{\vee})\;\;\; \mathrm{from}\; \mathrm{Lemma}\; \ref{Shapiro's lemma}\\
&=\mathrm{Ext}^{n}_{\Cg\text{-}\mathrm{mod}_{\mathrm{wt}}}(\mathbb{C}_0,M^{\vee}\otimes_{\C} N^{\vee})\;\;\;\mathrm{from}\; \mathrm{Lemma}\; \ref{extensions}\; (2)\\
&=\mathrm{Ext}^{n}_{\Cg\text{-}\mathrm{mod}^{\mathrm{int}}}(\mathbb{C}_0,M^{\vee}\otimes_{\C} N^{\vee})\;\;\;\mathrm{from}\; \mathrm{Lemma}\; \ref{extensions}\; (1)\\
&=\mathrm{Ext}^{n}_{\Cg\text{-}\mathrm{mod}^{\mathrm{int}}}(M,N^{\vee})\;\;\;\mathrm{from}\; \mathrm{Remark}\; \ref{extension dual2}.
\end{align*}
\end{proof}
\begin{proof}[Proof of Theorem \ref{ext}]
If we set $M=W(\lambda)\otimes_{\C}\C_{m\delta}$ and $N=W(\mu)_{loc}$ in Theorem \ref{compare extension}, then we obtain Theorem \ref{ext}.
\end{proof}
\begin{cor}\label{compare ext-pairing}
For each $f,$ $g\in \C(\!(q^{1/2})\!)[\mathring{P}]^{\mathring{W}},$ we have $$\langle f,g\rangle_{int}=\langle f,g\rangle_{\mathrm{Ext}}.$$
\end{cor}
\begin{proof}
From Theorem \ref{compare extension}, we have $$\langle \mathrm{gch}\;W(\lambda),\mathrm{gch}\;W(\mu)_{loc}\rangle_{int}=\langle \mathrm{gch}\;W(\lambda),\mathrm{gch}\;W(\mu)_{loc}\rangle_{\mathrm{Ext}}$$ for each $\lambda,$ $\mu\in \mathring{P}_+.$ Since $\{\mathrm{gch}\;W(\lambda)\}_{\lambda\in \mathring{P}_+}$ and $\{\mathrm{gch}\;W(\lambda)_{loc}
\}_{\lambda\in \mathring{P}_+}$ are $\C(\!(q^{1/2})\!)$-basis of $ \C(\!(q^{1/2})\!)[\mathring{P}]^{\mathring{W}},$ we obtain the assertion.
\end{proof}
\subsubsection{Demazure-Joseph functor}
For each $i=0,...,l$, let $\mathfrak{sl}(2,i)$ be a Lie subalgebra of $\g$ isomorphic to $\mathfrak{sl}_2$ corresponding to $\pm\al_{i}$ and $\mathfrak{p}_{i}:=\bo_{-}+\mathfrak{sl}(2,i)$. For each $i=0,...,l$ and a $\bo_{-}$-module $M$ with semisimple $\h$-action, $\mathcal{D}_i(M)$ is the unique maximal $\mathfrak{sl}(2,i)$-integrable quotient of $U(\mathfrak{p}_i)\underset{U(\bo_{-})}{\otimes}M$. Then $\mathcal{D}_i$ defines a functor called Demazure-Joseph functor (\cite{Jos}).

\begin{thm}[\cite{Jos}]\label{demazurefunctor}
For each $i=0,...,l$ and a $\h$-semisimple $\bo_{-}$-module $M$, the following hold:
\item $\mathrm{(1)}$ The functors $\{\mathcal{D}_i\}_{i=0,...,l}$ satisfy braid relations of $W$;
\item $\mathrm{(2)}$ There is a natural transformation $\mathrm{Id}\to\mathcal{D}_{i}$; 
\item $\mathrm{(3)}$ If $M$ is an $\mathfrak{sl}(2,i)$-integrable $\mathfrak{p}_i$-module then $\mathcal{D}_{i}(M)\cong M$;
\item $\mathrm{(4)}$ If $N$ is an $\mathfrak{sl}(2,i)$-integrable $\mathfrak{p}_i$-module then $\mathcal{D}_i(M\otimes N)\cong\mathcal{D}_i(M)\otimes N$;
\item $\mathrm{(5)}$ The functor $\df_i$ is right exact.
\end{thm}

For a reduced expression $w=s_{i_1}s_{i_2}\cdots s_{i_n}\in W$, we define $$\mathcal{D}_w:=\mathcal{D}_{i_1}\circ\mathcal{D}_{i_2}\circ\cdots\circ\mathcal{D}_{i_n}.$$ This is well-defined by Theorem $\ref{demazurefunctor}$ (1). 
\begin{thm}\label{thin demazure joseph functor}
For each $\Lambda\in P_+$, $w\in W$ and $i\in \{0,...,l\}$, we have
$$\mathcal{D}_i(D_{w\Lambda})=\begin{cases}D_{w\Lambda}&(w\geq s_iw)\\ D_{s_iw\Lambda}&(w<s_iw).
\end{cases}$$
\end{thm}
\begin{proof}
By Lemma \ref{extremal vector} and the PBW theorem, $D_{w\Lambda}$ has an integrable $\mathfrak{sl}_{2(i)}$-action if $w\geq s_iw$. Hence Theorem \ref{demazurefunctor} (3) implies $\mathcal{D}_i(D_{w\Lambda})=D_{w\Lambda}$ if $w\geq s_iw$. If $w<s_iw$, then $D_{s_iw\Lambda}$ is a $\mathfrak{p}_i$-module with an integrable $\mathfrak{sl}(2,i)$ action by Lemma \ref{extremal vector} and the PBW theorem, and we have an inclusion $D_{w\Lambda}\to D_{s_iw\Lambda}.$ Hence we have a morphism of $\mathfrak{p}_i$-module $U(\mathfrak{p_i})\underset{U(\bo_-)}{\otimes}D_{w\Lambda}\to D_{s_iw\Lambda}.$ This morphism is surjective since $D_{s_iw\Lambda}$ is generated by a vector with its weight $w\Lambda$ as $\mathfrak{p}_i$-module. Therefore we obtain a surjection $\mathcal{D}_i(D_{w\Lambda})\to D_{s_iw\Lambda}$ by taking a maximal $\mathfrak{sl}(2,i)$-integrable quotient. By \cite[Proposition 3.3.4]{Kas}, we have $\mathrm{gch}\;\mathcal{D}_i(D_{w\Lambda})=\mathrm{gch}\;D_{s_iw\Lambda}$. Hence the above surjection is an isomorphism. 
\end{proof}
We set $\df_i^{\#}:=\vee\circ\df_i\circ\vee.$

\begin{prop}[\cite{FKM} Proposition 5.7]{\label{dual}}
For each $i=0,1,...,l$, $n\in \mathbb{Z}_+,$ $M\in \mathfrak{B}^{\prime},$ $N\in \mathfrak{B}_0,$ we have
$$\mathrm{Ext}^{n}_{\mathfrak{B}}(\df_i(M),N)\cong\mathrm{Ext}^{n}_{\mathfrak{B}}(M,\df_i^{\#}(N))\;\;\; n\in \mathbb{Z}_+.$$
\end{prop}
\subsubsection{Realization of global Weyl module}

For each $\lambda\in \mathring{P}_{+}$, we define $$\mathrm{Gr}^{\lambda}D:=D^{\lambda}/\underset{\lambda\succ\mu, \mu\notin\mathring{W}\lambda}{\sum}D^{\mu}.$$ From the PBW theorem and Lemma \ref{extremal vector}, $D^{\lambda}$ and $\underset{\lambda\succ\mu, \mu\notin\mathring{W}\lambda}{\sum}D^{\mu}$ are stable under the action of $\Cg.$ Hence $\mathrm{Gr}^{\lambda}D$ admits a $\Cg$-module structure.

\begin{prop}\label{filtration}
Let $\lambda\in \mathring{P}_+,$ Then $\mathrm{Gr}^{\lambda}D$ has a filtration of $\bo_{-}$-submodules 
$$\{0\}=F_0\subset F_{1}\subset F_{2}\subset \cdots \subset F_{N-1}\subset F_{N}=\mathrm{Gr}^{\lambda}D$$ such that $$\{F_{i}/F_{i-1}\}_{i=1,...,N}=\{\D^{\mu}\}_{\mu\in \mathring{W}\lambda}.$$
\end{prop}
\begin{proof}
Let $\geq^{\prime}$ be a total order on $W$ such that if $w\geq v$ then $w\geq^{\prime} v$. For each $w\geq \pi_{\lambda}$, define
$$F_{w}:=(\underset{v\geq^{\prime}w}{\sum}D^{w\Lambda_{0}}+\underset{\lambda\succ\mu, \mu\notin\mathring{W}\lambda}{\sum}D^{\mu})/\underset{\lambda\succ\mu, \mu\notin\mathring{W}\lambda}{\sum}D^{\mu}.$$ This is a $\bo_{-}$-submodule of $\mathrm{Gr}^{\lambda}D$ and 
$$F_{w}\subseteq F_{v}\;\;\;\text{if}\;\;\;w\geq^{\prime}v.$$ By Corollary $\ref{quotient of demazure}$, $\{F_{w}\}_{w\in W}$ gives the assertion.
\end{proof}
\begin{lem}\label{graded character weyl integrable}
We have the following equality of graded characters. $$\mathrm{gch}\;L(\Lambda_{0})=\underset{\lambda\in \mathring{P}_{+}}{\sum}q^{(\lambda\mid\lambda)/2}\mathrm{gch}\;W(\lambda).$$
\end{lem}
\begin{proof}
Let $\lambda\in \mathring{P}_+$ and $k\in \mathbb{Z}_+$. By Theorem \ref{demazure=weyl}, $$\mathrm{Ext}_{\Cg\text{-}\mathrm{mod}^{\mathrm{int}}}^{k}(L(\Lambda_0),(\mathbb{C}_{-(\lambda\mid\lambda)\delta/2}\otimes_{\C} W(\lambda)_{loc})^{\vee})=\mathrm{Ext}_{\Cg\text{-}\mathrm{mod}^{\mathrm{int}}}^{k}(L(\Lambda_0),D_{\lambda}^{\vee}).$$ Applying Theorem \ref{thin demazure joseph functor} and Proposition \ref{dual} repeatedly, we have $$\mathrm{Ext}_{\Cg\text{-}\mathrm{mod}^{\mathrm{int}}}^{k}(L(\Lambda_0),D_{\lambda}^{\vee})=\mathrm{Ext}_{\Cg\text{-}\mathrm{mod}^{\mathrm{int}}}^{k}(L(\Lambda_0),D_{0}^{\vee})\;\;\;k\in\mathbb{Z}_+.$$ Here $D_{0}$ is isomorphic to the trivial $\Cg$-module $\mathbb{C}_{\Lambda_0}$ with its weight $\Lambda_0$. By \cite[Theorem 3.6]{HK}, We have a projective resolution of a $\Cg$-module $$\cdots \to P^1\to P^0 \to \mathbb{C}_{\Lambda_0}\to 0,$$ where $P^{n}=\bigoplus_{w\in W_0,\; l(w)=n}P(w\circ 0+\Lambda_0)_{\mathrm{int}}.$ Since $\mathrm{dim}_{\mathbb{C}}\;\mathrm{Hom}_{\Cg}(P^{n},\mathbb{C}_{\Lambda_0})=\delta_{0,n}$, we obtain $$\mathrm{dim}_{\mathbb{C}}\mathrm{Ext}_{\Cg\text{-}\mathrm{mod}^{\mathrm{int}}}^{k}(L(\Lambda_0),(\mathbb{C}_{-(\lambda\mid\lambda)\delta/2}\otimes_{\C} W(\lambda)_{loc})^{\vee})=\delta_{0,k}\;\;\;k\in\mathbb{Z}_{+}.$$ Therefore, we have $$\langle \mathrm{gch}\;L(\Lambda_{0}),\mathrm{gch}\;(\mathbb{C}_{-(\lambda\mid\lambda)\delta/2}\otimes_{\C} W(\lambda)_{loc})\rangle_{int}=1.$$ By Corollary \ref{characterequality2}, the set of graded characters $\{\mathrm{gch}\;W(\lambda)_{loc}\}_{\lambda\in \mathring{P}_+}$ is an orthogonal $\mathbb{C}(\!(q^{1/2})\!)$-basis of $\mathbb{C}(\!(q^{1/2})\!)[\mathring{P}].$ Hence Corollary \ref{orthogonality} implies the assertion.
\end{proof}
If a $\bo_{-}$-module $M$ admits a finite sequence of $\bo_-$-submodules such that every successive quotient is isomorphic to some $\D^{\mu}$ ($\mu\in P$), then we say $M$ is filtered by Demazure slices. Let $f$, $g\in \C(\!(q^{1/2})\!)[\mathring{P}]$. Here we make a convention that $f\geq g$ means all the coefficients of $f-g$ belong to $\mathbb{Z}_+$.
\begin{thm}\label{realization}
For each $\lambda\in \mathring{P}_+$, the global Weyl module $W(\lambda)\otimes_{\C}\C_{\Lambda_0}$ is isomorphic to $\mathrm{Gr}^{\lambda}D$ as $\Cg$-module. In particular, $W(\lambda)\otimes_{\C}\C_{\Lambda_0}$ is filtered by Demazure slices and each $\D^{\mu}$ $(\mu\in \mathring{W}\lambda)$ appears exactly once as a successive quotient.
\end{thm}
\begin{proof}
First, we show that there exists a surjection $W(\lambda)\otimes_{\C}\C_{\Lambda_0}\to\mathrm{Gr}^{\lambda}D$. Let $v_{\lambda}\in \mathrm{Gr}^{\lambda}D$ be the nonzero cyclic vector with its weight $\lambda+\Lambda_0-\frac{(\lambda|\lambda)}{2}\delta.$ We check $v_{\lambda}$ satisfies Definition $\ref{global weyl}$ (1), (2), (3). The condition (1) is trivial from the definition of $v_{\lambda}$. Since $L(\Lambda_0)$ is an integrable $\g$-module, $v_{\lambda}$ is an extremal weight vector. This implies the condition (2). We check the condition (3) in the sequel. Since $\langle\lambda,\check{\al}\rangle\geq0$ and $v_{\lambda}$ is an extremal weight vector, we have $e_{\al}v_{\lambda}=0$ for $\al\in \mathring{\De}_{+}$. For each $\mu\in \mathring{P},$ we set $|\mu\rangle:=1\in R_{\mu}.$ For each $\beta=\al+n\delta \in \De$ with $n\in -\mathbb{Z}_+/2$ and $\al \in \mathring{\De}_{+}\cup \frac{1}{2}\mathring{\De}_{l+},$ we have $v:=e_{\al+n\delta}\vl\in U(\Cg_{im})|\lambda+\al\rangle$ by Theorem \ref{frenkelkac}. Since $U(\mathring{\g})v$ is finite dimensional,  $U(\mathring{\g})v$ has a highest weight vector whose weight is $\nu$. Then, $v\in U(\Cg_{im})U(\mathring{\mathfrak{n}}_-)|\nu\rangle\subset D^{\nu}$. Hence $\nu-\lambda\in \mathring{Q}_+^{\prime}.$ Since $\lambda$ and $\nu$ is dominant, we have $\lambda\succ\nu.$ Therefore $D^{\nu}$ is $0$ in $\gr$ as $|\nu\rangle\in D^{\nu}$. This implies $v=0$ and we have the desired surjection. In particular, we have an inequality
$$
q^{(\lambda,\lambda)/2}\mathrm{gch}\;W(\lambda)\geq\mathrm{gch}\;\gr.$$
On the other hand, we have, 
$$\mathrm{gch}\;L(\Lambda_{0})=\underset{\lambda\in \mathring{P}_{+}}{\sum}\mathrm{gch}\;\mathrm{Gr}^{\lambda}D$$ and 
$$\mathrm{gch}\;L(\Lambda_{0})=\underset{\lambda\in \mathring{P}_{+}}{\sum}q^{(\lambda\mid\lambda)/2}\mathrm{gch}\;W(\lambda)$$ by Lemma \ref{graded character weyl integrable}.
 Thus the above inequality is actually an equality and the assertion follows. 
\end{proof}

\subsection{Extensions between $\D^{\lambda}$ and $D_{\mu}$}
\subsubsection{Demazure-Joseph functor and Demazure slices}
\begin{thm}\label{thick demazure joseph functor}
For each $w\in W$ and $i\in\{0,...,l\}$, we have the following:
\begin{equation*}
\mathcal{D}_{i}(D^{w})=
\begin{cases}
D^{s_iw}&if\;s_iw<w\\
D^{w}&if\;s_iw\geq w.
\end{cases}
\end{equation*}
\end{thm}
\begin{proof}
The proof is the same as proof of  Theorem \ref{thin demazure joseph functor} using the analog of \cite[Proposition 3.3.4]{Kas} for thick Demazure modules (cf. \cite[\S4]{Kas}).
\end{proof}

For each $c\in \mathring{P}$, let $W(c)$ be the image of $D^{c}$ in $\mathrm{Gr}^{c_+}D$. Here $c_{+}$ is a unique dominant integrable weight in $\mathring{W}c$. From Theorem \ref{realization}, the global Weyl module is isomorphic to the image of $D^{c_+}$ as $\Cg^{\prime}$-module. Hence this notation is consistent with the previous notation and we use the same notation.

\begin{prop}[\cite{CK} Proposition 4.13]\label{generalizedweyl}
For each $c\in \mathring{P}$ and $i\in\{1,...,l\}$, we have
\begin{equation*}
\mathcal{D}_i(W(c))=
\begin{cases}
W(s_ic)&(s_ic\succeq c)\\
W(c)&(s_ic\nsucceq c)
\end{cases}.
\end{equation*}
\end{prop}
\begin{proof}
We set $M_{c}:=\underset{c_{-}\succ a}{\sum}D^{a}.$ We have a short exact sequence
$$0\to M_c\to D^{c}+M_c\to W(c)\to 0.$$
The module $M_c$ is invariant under $\df_i$ by Theorem \ref{thick demazure joseph functor},  and hence we obtain a following exact sequence 
$$\mathbb{L}^{-1}\mathcal{D}_i(W(c))\to M_c\to D^{c^{\prime}}+M_c\to \df_i(W(c))\to 0.$$
Here 
\begin{equation*}
c^{\prime}=
\begin{cases}
s_ic&(s_ic\succeq c)\\
c&(s_ic\nsucceq c)
\end{cases}
\end{equation*} and $\mathbb{L}^{\bullet}\mathcal{D}_i$ is the left derived functor of $\mathcal{D}_i$. By Theorem \ref{demazurefunctor} (2), we have the following commutative diagram \[
  \begin{CD}
     M_c @>>> D^c+M_c \\
  @|    @VVV \\
     M_c  @>>>  D^{c^\prime}+M_c
  \end{CD}.
\] Since $L(\Lambda_0)$ is completely reducible as a $\mathfrak{sl}(2,i)$-module and $D^{c}+M_c$ is a $\bo_-$-submodule of $L(\Lambda_0)$, the above morphism $D^{c}+M_c\to D^{c^{\prime}}+M_c$ is injective by \cite[Lemma 2.8 (1)]{Jos}. Hence $M_c\to D^{c^{\prime}}+M_c$ is injective. Therefore we obtain $\df_i(W(c))\cong (D^{c^{\prime}}+M_c)/M_c$ from the above exact sequence.
\end{proof}

\begin{prop}[\cite{CK} Corollary 4.15]\label{exactsequence}
Let $i\in\{0,1,...,l\}$ and $c\in\mathring{P}$. If $s_ic\succ c$, then we have an exact sequence
$$0\to\D^c\to\df_i(\D^c)\to\D^{s_ic}\to 0$$ and $\df_{i}(\D^{s_i c})=\{0\}$.
\end{prop}
\begin{proof}
We set $S:=\{w\in W|w\nleq\pi_c,\;s_iw\nleq\pi_c\}$ and $M:=\underset{w\in S}{\sum}D^{w}$. Then we have $D^{c}\cap M=\underset{\pi_c<w}{\sum}D^{w}$. Hence we have an exact sequence
$$0\to M\to D^{c}+M\to \D^{c}\to 0.$$ As $s_i(S)\subset S$, we have $\df_i(M)\cong M$. By the same argument as in the proof of Proposition \ref{generalizedweyl}, applying $\df_i$, we obtain
$$0\to M\to D^{s_ic}+M\to \df_i(\D^{c})\to 0.$$ In particular, we have $$\D^c\cong (D^c+M)/M\;\;\mathrm{and}\;\;\df_i(\D^{c})\cong (D^{s_ic}+M)/M.$$
Hence we have
$$0\to \D^{c}\to\df_{i}(\D^{c})\to (D^{s_ic}+M)/(D^{c}+M)\to0.$$
Here $(D^{s_ic}+M)/(D^{c}+M)\cong D^{s_ic}/(D^{s_ic}\cap (D^{c}+M))$ is isomorphic to $\D^{s_ic}$ since $D^{s_ic}\cap (D^{c}+M)=\sum_{w>s_i\pi_c}D^w$. Hence the first assertion follows.
Applying $\df_i$ to the last exact sequence, from right exactness of $\df_i$, we have an exact sequence
$$\df_i(\D^{c})\to \df_i^2(\D^{c})\to\df_i(\D^{s_ic})\to0.$$ From Theorem \ref{demazurefunctor} (3), the above homomorphism $\df_i(\D^{c})\to \df_i^2(\D^{c})$ is an isomorphism.
This implies the second assertion.
\end{proof}

\subsubsection{Calculation of $\mathrm{Ext}_{\mathfrak{B}}^n(\D^{\lambda}\otimes_{\C}\C_{m\delta+k\Lambda_0},D_{\mu}^{\vee}$)}
The following theorem is an $A_{2l}^{(2)}$ version of \cite[Theorem 4.18]{CK}.
\begin{thm}\label{extduality}
For each $\lambda$, $\mu\in \mathring{P}$, $m\in\mathbb{Z}/2$ and $k\in \mathbb{Z}$, we have
$$\mathrm{dim}_{\mathbb{C}}\;\mathrm{Ext}_{\mathfrak{B}}^n(\D^{\lambda}\otimes_{\C}\C_{m\delta+k\Lambda_0},D_{\mu}^{\vee})=\delta_{n,0}\delta_{m,0}\delta_{k,0}\delta_{\lambda,\mu}\;\;\;n\in\Z_+.$$
\end{thm}
\begin{proof}
By comparing the level, the extension vanishes if $k\neq0$.
We prove the assertion by induction on $\mu$ with respect to $\succ$. By Theorem \ref{demazurefunctor} (3), we have $\df_{w}(D_{0})=D_{0}$ for all $w\in \mathring{W}$. If $\lambda$ is not anti-dominant, then there exists $i\in\{1,...,l\}$ such that $s_i\lambda>\lambda$. Hence 
\begin{eqnarray*}
\mathrm{Ext}_{\mathfrak{B}}^n(\D^{\lambda}\otimes_{\C}\C_{m\delta+k\Lambda_0},D_{0}^{\vee})&=&\mathrm{Ext}_{\mathfrak{B}}^n(\D^{\lambda}\otimes_{\C}\C_{m\delta+k\Lambda_0},\df^{\#}_i(D_{0}^{\vee}))\\&=&\mathrm{Ext}_{\mathfrak{B}}^n(\mathcal{D}_i(\D^{\lambda}\otimes_{\C}\C_{m\delta+k\Lambda_0}),D_{0}^{\vee})\\&=&\{0\}.
\end{eqnarray*} Here we used Proposition \ref{dual} in the second equality and Proposition \ref{exactsequence} in the third equality.
If $\lambda$ is anti-dominant, then we have $\df_{w_0}(\D^{\lambda})=W(\lambda_+)\otimes_{\C}\C_{\Lambda_0}$ for the longest element $w_0$ of $\mathring{W}$ by Proposition $\ref{generalizedweyl}$. Hence we have $\mathrm{Ext}_{\mathfrak{B}}^n(\D^{\lambda}\otimes_{\C}\C_{m\delta+k\Lambda_0},D_{0}^{\vee})=\mathrm{Ext}_{\mathfrak{B}}^n(W(\lambda_+)\otimes_{\C}\C_{m\delta+k\Lambda_0},W(0)_{loc}^{\vee})$ by Theorem $\ref{demazure=weyl}.$ From Theorem $\ref{ext}$, the assertion follows in this case. \\Let $s_i\mu\succ \mu$. We set $\D_{\lambda}^{\prime}:=\D_{\lambda}\otimes_{\C}\C_{m\delta+k\Lambda_0}$ for $\lambda\in P.$ By Proposition $\ref{exactsequence}$, we have the following exact sequence
\begin{alignat*}{2}
0&
\to \mathrm{Ext}_{\mathfrak{B}}^0(\D_{s_i\lambda}^{\prime},D_{\mu}^{\vee})\to\mathrm{Ext}_{\mathfrak{B}}^0(\df_i(\D_{\lambda}^{\prime}),D_{\mu}^{\vee})\to\mathrm{Ext}_{\mathfrak{B}}^0(\D_{\lambda}^{\prime},D_{\mu}^{\vee})\to\\
        \cdots &\rightarrow \mathrm{Ext}_{\mathfrak{B}}^n(\D_{s_i\lambda}^{\prime},D_{\mu}^{\vee})\to\mathrm{Ext}_{\mathfrak{B}}^n(\df_{i}(\D_{\lambda}^{\prime}),D_{\mu}^{\vee})\to\mathrm{Ext}_{\mathfrak{B}}^n(\D_{\lambda}^{\prime},D_{\mu}^{\vee})\to &\cdots.
\end{alignat*}From Theorem \ref{thin demazure joseph functor} and Proposition $\ref{dual}$, we have $$\mathrm{Ext}_{\mathfrak{B}}^n(\df_{i}(\D_{\lambda}\otimes_{\C}\C_{m\delta+k\Lambda_0}),D_{\mu}^{\vee})\cong \mathrm{Ext}_{\mathfrak{B}}^n(\D_{\lambda}\otimes_{\C}\C_{m\delta+k\Lambda_0},\df^{\#}_i(D_{\mu}^{\vee}))\cong\mathrm{Ext}_{\mathfrak{B}}^n(\D_{\lambda}\otimes_{\C}\C_{m\delta+k\Lambda_0},D_{s_i\mu}^{\vee}).$$ Therefore, the assertion follows from the induction hypothesis and the long exact sequence.

\end{proof}

\begin{cor}\label{extorthogonality}
For each $\lambda$, $\mu\in\mathring{P}$, we have
$$\langle\mathrm{gch}\;\D^{\lambda},q^{\frac{(b|b)}{2}}\bar{E}_{\mu}(X^{-1},q^{-1})\rangle_{\mathrm{Ext}}=\delta_{\lambda,\mu}.$$
\end{cor}
\begin{proof}
By Theorem $\ref{extduality}$, we have $$\langle\mathrm{gch}\;\D^{\lambda},\mathrm{gch}\;D_{\mu}\rangle_{\mathrm{Ext}}=\delta_{\lambda,\mu}.$$
Using Theorem $\ref{characterequality}$, we obtain the assertion.
\end{proof}
\begin{cor}\label{character of demazure slice}
For each $\lambda\in\mathring{P}$, we have
$$\mathrm{gch}\;\D^{\lambda}=q^{\frac{(b|b)}{2}}{E}^{\dag}_{\lambda}(X^{-1},q^{-1})/\langle \bar{E}_{\lambda},E^{\dag}_{\lambda}\rangle_{\mathrm{Ext}}.$$
\end{cor}
\begin{proof}Since $\{{E}^{\dag}_{\lambda}(X^{-1},q^{-1})/\langle \bar{E}_{\lambda},E^{\dag}_{\lambda}\rangle_{\mathrm{Ext}}\}_{\mu\in\mathring{P}}$ is a $\mathbb{C}(\!(q^{1/2})\!)$-basis of $\mathbb{C}(\!(q^{1/2})\!)[\mathring{P}]$, we have $$\mathrm{gch}\;\D^{\lambda}=\sum_{\mu\in\mathring{P}}a_{\mu}{E}^{\dag}_{\lambda}(X^{-1},q^{-1})/\langle \bar{E}_{\lambda},E^{\dag}_{\lambda}\rangle_{\mathrm{Ext}}$$ for some $a_{\mu}\in\mathbb{C}(\!(q^{1/2})\!)$. Since $\langle-,-\rangle_{nonsym}^{\prime}|_{t=0}=\langle-,-\rangle_{\mathrm{Ext}},$ and $\{E_{\lambda}(X,q)\}_{\lambda\in\mathring{P}}$ are orthogonal with respect to $\langle-,-\rangle_{nonsym}^{\prime}$ each other, we have $$\langle\bar{E}_{\lambda}(X^{-1},q^{-1}), E^{\dag}_{\mu}(X^{-1},q^{-1})\rangle_{\mathrm{Ext}}/\langle \bar{E}_{\lambda},E^{\dag}_{\lambda}\rangle_{\mathrm{Ext}}=\delta_{\lambda,\mu}.$$ Hence we have $\langle\mathrm{gch}\;\D^{\lambda},\mathrm{gch}\;D_{\mu}\rangle_{\mathrm{Ext}}=a_{\mu}$ by Theorem \ref{characterequality}. Therefore the assertion follows from Corollary \ref{extorthogonality}.
\end{proof}

\section{Weyl module for special current algebra of $A_{2l}^{(2)}$}
We continue to work in the setting of the previous section.
\subsection{Special current algebra of $A_{2l}^{(2)}$}
In this subsection, we refer for general terminologies to \cite[Chapter 2]{FK}, \cite[\S2.2]{FM} and \cite[Appendix]{Car}.
We set $$\mathring{\h}^{\dag}:=\bigoplus_{i=0}^{l-1}\mathbb{C}\al_i,\;\;\mathring{\De}^{\dag}:=\De\cap \mathring{Q}^{\dag}\;\; \mathrm{and}\;\;\mathring{\g}^{\dag}:=(\underset{\al\in \mathring{\De}^{\dag}}{\bigoplus}\g_{\al})\oplus\mathring{\h}^{\dag}.$$ Then $\mathring{\g}^{\dag}$ is a finite dimensional simple Lie algebra of type $B_l.$ The subalgebra $\mathring{\h}^{\dag}$ is a Cartan subalgebra of $\mathring{\g}^{\dag},$ and $\mathring{\De}^{\dag}$ is the set of roots of $\mathring{\g}^{\dag}$ with respect to $\mathring{\h}^{\dag}$. Using the standard basis $\nu_1,...,\nu_l$ of $\mathbb{R}^{l}$, we have :
$$\mathring{\De}^{\dag}=\{\pm(\nu_{i}\pm\nu_{j}),\;\pm\nu_{i}|\;i,j=1,...,l\}.$$
We denote the set of short roots of $\mathring{\g}^{\dag}$ by $\mathring{\De}^{\dag}_{s}$ and the set of long roots of $\mathring{\g}^{\dag}$ by $\mathring{\De}^{\dag}_l.$ Let $\{\al_1^{\dag},...,\al_l^{\dag}\}$ be a set of simple roots of $\mathring{\g}^{\dag}.$ We set $\mathring{Q}^{\dag}_+:=\bigoplus_{i=1,...,l}\mathbb{Z}_+\al^{\dag}.$ We have $$\De_{re}=(\mathring{\De}^{\dag}+\mathbb{Z}\delta)\cup(2\mathring{\De}^{\dag}_s+(2\mathbb{Z}+1)\delta).$$
The special current algebra $\Cg^{\dag}$ is the maximal parabolic subalgebra of $\g$ that contains $\mathring{\g}^{\dag}$. We have $\Cg^{\dag}=\mathring{\g}^{\dag}+\bo_{-}.$ We set $$\Cg_{im}^{\dag}:=\Cg_{im},\;\; \Cg^{\dag\prime}:=[\Cg^{\dag},\Cg^{\dag}]$$ and $$\mathfrak{Cn}_{+}^{\dag}:=\underset{\al\in (\mathring{\De}^{\dag}_+-\mathbb{Z}_+\delta)\cup(2\mathring{\De}^{\dag}_{s+}-(2\mathbb{Z}_++1)\delta)}{\bigoplus}\g_{\al}$$ Let $\mathring{P}^{\dag}$ be the integral weight lattice of $\mathring{\g}^{\dag}$ and $\mathring{P}_{+}^{\dag}$ be the set of dominant integral weights of $\mathring{\g}^{\dag}.$ Let $\varpi_{i}^{\dag}$ $(i=1,...,l)$ be the fundamental weights of $\mathring{\g}^{\dag}$. We identify  $\mathring{P}^{\dag}$ and $\mathbb{Z}\varpi_1^{\dag}\oplus \cdots \oplus \mathbb{Z}\varpi_{l-1}^{\dag}\oplus \mathbb{Z}\varpi_l^{\dag}$ by $\varpi_i^{\dag}=\Lambda_{l-i}-\Lambda_l$ for $i\neq l$ and $\varpi_l^{\dag}=\Lambda_0-\Lambda_l/2$. Let $\mathring{W}^{\dag}$ be the subgroup of $W$ generated by $\{s_{\al}\}_{\al\in \mathring{\De}^{\dag}}.$ 

\subsection{Realization of $\Cg^{\dag}$}
We refer to \cite[\S4.6]{CIK} in this subsection. Let $X_{i,j}$ be a $(2l+1)\times (2l+1)$ matrix unit whose $ij$-entry is one. We set $H_i=X_{i,i}-X_{i+1,i+1}$ $(i=1,...,2l)$. The Lie algebra $\sll$ is spaned by $X_{i,j}$ $(i\neq j)$ and $H_i$ $(i=1,...,2l)$. The assignment
$$X_{i,i+1}\to X_{2l+1-i,2l+2-i},\;X_{i+1,i}\to X_{2l+2-i,2l+1-i}$$ extends on $\sll$ as a Lie algebra automorphism. We write this automorphism by $\sigma$. Let $L(\sll)=\sll\otimes_{\C} \C[t^{\pm1}]$ be the loop algebra corresponding to $\sll$ and extend $\sigma$ on $L(\sll)$ by $\sigma(X\otimes f(t))=\sigma(X)\otimes f(-t)$. We denote the fixed point of $\sigma$ in $\sll\otimes_{\C}\C[t]$ by $(\sll\otimes_{\C}\C[t])^{\sigma}.$

\begin{prop}[see \cite{Kac} Theorem 8.3]
The Lie algebra $(\sll\otimes_{\C}\C[t])^{\sigma}$ is isomorphic to $\Cg^{\dag\prime}$.
\end{prop}
 
\subsection{Weyl module for $\Cg^{\dag}$}

\begin{defi}\label{special global weyl}
For each $\lambda\in \mathring{P}^{\dag}_+$, the global Weyl module is a cyclic $\Cg^{\dag}$-module $W(\lambda)^{\dag}$ generated by $v_{\lambda}$ that satisfies the following relations:
\item $\mathrm{(1)}$ $hv_{\lambda}=\lambda(h)v_{\lambda}$ for each $h\in \h$;
\item $\mathrm{(2)}$ $e_{-\alpha}^{\langle\lambda,\check{\al}\rangle+1}v_{\lambda}=0$ for each $\al\in\mathring{\De}_{+}^{\dag}$;
\item $\mathrm{(3)}$ $\mathfrak{Cn}_+^{\dag}v_{\lambda}=0.$
\end{defi}

\begin{defi}\label{special local weyl}For each $\lambda\in \mathring{P}$, the local Weyl module is a cyclic $\Cg^{\dag}$-module $W(\lambda)_{loc}^{\dag}$ generated by $v_{\lambda}$ satisfies relations $\mathrm{(1)}$, $\mathrm{(2)}$, $\mathrm{(3)}$ of Definition \ref{special global weyl} and 
\item $\mathrm{(4)}$ $Xv_{\lambda}=0$ for each $X\in \Cg_{im}^{\dag}$.
\end{defi}

\begin{thm}[\cite{FK} Corollary 6.0.1 and \cite{FM} Corollary 2.19]\label{dimension formula}
For each $\lambda\in \mathring{P}_+^{\dag}$, we have
\item $\mathrm{(1)}$ If $\lambda=\sum_{i=1}^{l-1}m_i\varpi_i^{\dag}+(2k-1)\varpi_l^{\dag}$, then 
$$\mathrm{dim}_{\C}\;W(\lambda)^{\dag}_{loc}=\left(\prod_{i=1}^{l-1}\binom{2l+1}{i}^{m_i}\right)\binom{2l+1}{l}^{k-1}2^l;$$
\item $\mathrm{(2)}$ If $\lambda=\sum_{i=1}^{l-1}m_i\varpi_i^{\dag}+2m_l\varpi_l^{\dag}$, then
$$\mathrm{dim}_{\C}\;W(\lambda)^{\dag}_{loc}=\prod_{i=1}^{l}\binom{2l+1}{i}^{m_i}.$$
\end{thm}
\subsection{The algebra $\mathbf{A}_{\lambda}$}
Let $\lambda\in \mathring{P}^{\dag}_+$. We set $$\mathrm{Ann}(v_{\lambda}):=\{X\in U(\Cg_{im}^{\dag})\;|\;Xv_{\lambda}=0\}\;\; \mathrm{and}\;\; \mathbf{A}_{\lambda}:=U(\Cg_{im}^{\dag})/\mathrm{Ann}(v_{\lambda}),$$ where $v_{\lambda}$ is the cyclic vector of $W(\lambda)_{loc}^{\dag}$ in Definition \ref{special global weyl}.

\begin{prop}[\cite{CIK} \S7.2]
For each $\lambda\in \mathring{P}_{+}^{\dag}$, the algebra $\mathbf{A}_{\lambda}$ acts on $W(\lambda)^{\dag}$ by $$X.Yv_{\lambda}:=YXv_{\lambda}\;\;for\;\; X\in \mathbf{A}_{\lambda}\;and\;Y\in U(\Cg^{\dag\prime}).$$
\end{prop}
\subsubsection{Generator of $\mathbf{A}_{\lambda}$}
For $i=1,...,l-1$, we set $$h_{i,0}:=H_{i}+H_{2l+1-i},\;\;h_{i,1}:=H_{i}-H_{2l+1-i},$$ $$x_{i,0}:=X_{i,i+1}+X_{2l+1-i,2l+2-i},\;\;x_{i,1}:=X_{i,i+1}-X_{2l+1-i,2l+2-i},$$ $$y_{i,0}:=X_{i+1,i}+X_{2l+2-i,2l+1-i},\;\;y_{i,1}:=X_{i+1,i}-X_{2l+2-i,2l+1-i}$$ and $$h_{l,0}=2(H_{l}+H_{l+1}),\;\;h_{l,1}=H_{l}-H_{l+1},$$ $$x_{l,0}:=\sqrt{2}(X_{l,l+1}+X_{l+1,l+2}),\;\;x_{l,1}:=-\sqrt{2}(X_{l,l+1}-X_{l+1,l+2}),$$ $$y_{l,0}:=\sqrt{2}(X_{l+1,l}+X_{l+2,l+1}),\;\;y_{l,1}:=-\sqrt{2}(X_{l+1,l}-X_{l+2,l+1}).$$ The Lie algebra generated by $\{x_{i,0},\;y_{i,0},\;h_{i,0}\}_{i=1,...,l}$ is isomorphic to the simple Lie algebra of type $B_l,$ and $\{h_{i,0}\}_{i=1,...,l}$ is the set of its simple coroots \cite[Theorem 9.19]{Car}. We set $z_{l,1}:=\frac{1}{4}[y_{l,0},y_{l,1}].$  As in \cite[\S3.3]{CFS}, we define $p_{i,r}\in U(\Cg_{im}^{\dag})$ $(i=1,..,l\;and\;r\in \mathbb{Z}_{+})$ by 
$$\underset{r\in\mathbb{Z}_{+}}{\sum}p_{i,r}z^r:=\mathrm{exp}\left(-\sum_{k=1}^{\infty}\sum_{\varepsilon=0}^{1}\frac{h_{i,\varepsilon}\otimes t^{-2k+\varepsilon}}{2k-\varepsilon}z^{2k-\varepsilon}\right)$$ for $i\neq l$ and 
$$\underset{r\in\mathbb{Z}_{+}}{\sum}p_{l,r}z^r:=\mathrm{exp}\left(-\sum_{k=1}^{\infty}\frac{h_{l,0}/2\otimes t^{-2k}}{2k}z^{2k}+\sum_{k=1}^{\infty}\frac{h_{l,1}\otimes t^{-2k+1}}{2k-1}z^{2k-1}\right).$$
\begin{prop}\label{generator}
The algebra $U(\Cg_{im}^{\dag})$ is isomorphic to the polynomial ring $\C[p_{i,r} |i=1,...,l,\;r\in \mathbb{Z}_+]$.
\end{prop}
\begin{proof}
We have $\C[p_{i,r} |i=1,...,l,\;r\in \mathbb{Z}_+]\subset U(\Cg_{im}^{\dag})$. The set of generators of $U(\Cg_{im}^{\dag})$ is $\{h_{i,\varepsilon}\otimes t^{-2k+\varepsilon}|\;n\in \{1,..,l\}$, $k\in \mathbb{N}\;and\; \varepsilon\in \{0,1\}\}.$ It suffices to see that $h_{n,\varepsilon}\otimes t^{-2k+\varepsilon}\in\C[p_{i,r} |i=1,...,l,\;r\in \mathbb{Z}_+]$ for each $i\in \{1,..,l\}$, $k\in \mathbb{N}$ and $\epsilon\in \{0,1\}.$ We have $h_{i,1}\otimes t^{-1}=p_{i,1}$ up to a constant multiple. By definition, $p_{i,2k-\varepsilon}+(h_{i,\varepsilon}\otimes t^{-2k+\varepsilon})/(2k-\varepsilon)$ is an element of $\mathbb{Q}[h_{i,s}|s<2k-\varepsilon]$ if $i\neq l$, and $p_{l,2k-\varepsilon}-(-1)^{\varepsilon+1}(h_{l,\varepsilon}/2^{1-\varepsilon}\otimes t^{-2k+\varepsilon})/(2k-\varepsilon)$ is an element of $\mathbb{Q}[h_{l,s}|s<2k-\varepsilon].$ The assertion follows by induction on $2k-\varepsilon.$ 
\end{proof}
\begin{lem}[\cite{CFS} Lemma 3.2, Lemma 3.3 (iii) (b) and \cite{CP} Lemma 1.3 (ii)]\label{technical identity}Let $V$ be a $\Cg^{\dag}$-module and $v\in V$ be a nonzero vector such that $\mathfrak{Cn}_+v=0.$ We have the following: 
\item $\mathrm{(1)}$ For $i\neq l,$ we have $(x_{i,1}\otimes t^{-1})^{(r)}(y_{i,0})^{(r)}v=(-1)^rp_{i,r}v$ for $r\in \mathbb{N}$;
\item $\mathrm{(2)}$ We have $(x_{l,0})^{(2r)}(z_{l,1}\otimes t^{-1})^{(r)}v=(-1)^rp_{l,r}v$ for $r\in \mathbb{N}.$
\end{lem}
\begin{prop}\label{annihilator}
Let $\lambda\in \mathring{P}_{+}^{\dag},$ $i\in\{1,...,l-1\}$ and $v_{\lambda}$ be the cyclic vector of $W(\lambda)^{\dag}$ with its weight $\lambda$. We have $p_{i,r}v_{\lambda}=0$ for $r>\langle\lambda,\check{\al}_{i}^{\dag}\rangle$, and $p_{l,r}v_{\lambda}=0$ for $r>\lfloor\frac{\langle\lambda,\check{\al}_{l}^{\dag}\rangle }{2}\rfloor$.
\end{prop}
\begin{proof}
Definition \ref{special global weyl} (3) implies the set of $\mathring{\h}^{\dag}$-weights of $W(\lambda)^{\dag}$ is the subset of $\lambda-\mathring{Q}^{\dag}_{+}.$ From Definition \ref{special global weyl} (2) and Lemma \ref{technical identity} (1), we get $p_{i,r}v_{\lambda}=0$ for $r>\langle\lambda,\check{\al}_{i}^{\dag}\rangle.$ By Definition \ref{special global weyl} (2), $W(\lambda)^{\dag}$ is an $\mathring{\g}^{\dag}$-integrable module. Since the set of $\mathring{\h}^{\dag}$-weights of $W(\lambda)^{\dag}$ is contained in $\lambda-\mathring{Q}_+^{\dag},$ this implies $\lambda-k\al_l^{\dag}$ for $k>\langle \lambda,\check{\al}^{\dag}_l\rangle$ is not a weight of a vector of $W(\lambda)^{\dag}.$ Since $(z_{l,1}\otimes t)$ is a root vector corresponding to $2\al_l^{\dag}-\delta,$ we obtain $p_{l,r}v_{\lambda}=0$ for $r>\lfloor\frac{\langle\lambda,\check{\al}_{l}^{\dag}\rangle }{2}\rfloor.$  
\end{proof}
We set $$\mathbf{A}_{\lambda}^{\prime}:=\C[p_{i,r}|1\leq r\leq \langle\lambda,\check{\al}_{i}^{\dag}\rangle\;\mathrm{for}\;i\neq l,\;1\leq r\leq \lfloor\frac{\langle\lambda,\check{\al}_{l}^{\dag}\rangle }{2}\rfloor\;\mathrm{for}\;i=l].$$
\begin{cor}\label{surjection}
For each $\lambda\in \mathring{P}_+^{\dag},$ there exists a $\mathbb{C}$-algebra surjection $\mathbf{A}_{\lambda}^{\prime}\to\mathbf{A}_{\lambda}$.
\end{cor}
\begin{proof}
By Proposition \ref{annihilator}, we have $p_{i,r}$, $p_{l,k}\in \mathrm{Ann}(v_{\lambda})$ for each $r>\langle\lambda,\check{\al}_{i}^{\dag}\rangle$ ($i\neq l$) and each $k>\lfloor\frac{\langle\lambda,\check{\al}_{n}^{\dag}\rangle }{2}\rfloor$. Hence we have a surjection $\mathbf{A}_{\lambda}^{\prime}\to \mathbf{A}_{\lambda}$ by Proposition \ref{generator}.
\end{proof}
We set $\mathring{P}^{\dag\prime}_+:=\{\lambda \in \mathring{P}^{\dag}_+\;|\;\langle \lambda,\check{\al}_l\rangle\in 2\mathbb{Z}_+\}$.
\begin{thm}[\cite{CIK} \S5.6 and Theorem 1]\label{chari}
For each $\lambda\in \mathring{P}_{+}^{\dag\prime}$ and nonzero element $f\in\mathbf{A}_{\lambda}^{\prime},$ there exists a quotient of $W(\lambda)^{\dag}$ such that $f$ acts nontirivially on the image of the cyclic vector $v_{\lambda}$ of $W(\lambda)^{\dag}.$ In particular $\mathbf{A}_{\lambda}\cong \mathbf{A}_{\lambda}^{\prime}.$
\end{thm}

\begin{lem}[\cite{CIK} Lemma 5.4]\label{p}
For each $1\leq s\leq k$, let $V_s$ be representations of $\Cg^{\dag}$ and let $v_s$ be vectors  of $V_s$ such that $\mathfrak{Cn}_{+}^{\dag}v_s=0.$ We have
$$p_{i,r}(v_1\otimes \cdots \otimes v_k)=\sum_{r=j_1+\cdots+j_k,\;j_i\geq0}p_{i,j_1}v_1\otimes\cdots\otimes p_{i,j_k}v_k$$ for all $1\leq i \leq l$ and $r\in \mathbb{Z}_+$.
\end{lem}
\subsubsection{Dimension inequalities}For each maximal ideal $\mathbf{I}$ of $\mathbf{A}_{\lambda}$, we define $$W(\lambda,\mathbf{I})^{\dag}:=(\mathbf{A}_{\lambda}/\mathbf{I})\underset{\mathbf{A}_{\lambda}}{\otimes}W(\lambda)^{\dag}.$$ Let $U(\Cg_{im})_+$ be the argumentation ideal of $U(\Cg_{im})$ and $\mathbf{I}_{\lambda,0}$ be a maximal ideal of $\mathbf{A}_{\lambda}$ defined by $(U(\Cg_{im})_++\mathrm{Ann}(v_{\lambda}))/\mathrm{Ann}(v_{\lambda})$. 
\begin{prop}
For each $\lambda\in \mathring{P}^{\dag}_+$, we have $W(\lambda)_{loc}^{\dag}\cong W(\lambda,\mathbf{I}_{\lambda,0})^{\dag}$. 
\end{prop}
\begin{proof}
The assertion follows from Definition \ref{special local weyl} (4).
\end{proof} 
\begin{prop}[\cite{CIK} Proposition 6.4 and 6.5]\label{dimension inequality}Let $\lambda\in \mathring{P}_+^{\dag}$ and let $\mathbf{I}$ be a maximal ideal of $\mathbf{A}_{\lambda}.$
\item $\mathrm{(1)}$ If $\mu\in\mathring{P}^{\dag}_+$ satisfies $\lambda-\mu\in \mathring{P}^{\dag\prime}_{+}$, then we have $$\mathrm{dim}_{\C}\;W(\lambda,\mathbf{I})^{\dag}\geq \mathrm{dim}_{\C}\;W(\mu)_{loc}^{\dag}\left(\prod_{i=1}^{l-1}\binom{2l+1}{i}^{(\lambda-\mu)(\check{\al}_i)}\right)\binom{2l+1}{l}^{(\lambda-\mu)(\check{\al}_l/2)}.$$
\item $\mathrm{(2)}$ We have $$\mathrm{dim}_{\C}\;W(\lambda)_{loc}^{\dag}\geq \mathrm{dim}_{\C}\;W(\lambda,\mathbf{I})^{\dag}.$$
\end{prop}

\begin{cor}[\cite{CIK} Theorem 10 when $\lambda\in\mathring{P}^{\dag\prime}_+$]\label{dimension equality}
For each $\lambda\in \mathring{P}_{+}^{\dag}$ and each maximal ideal $\mathbf{I}$ of $\mathbf{A}_{\lambda}$, the dimension $\mathrm{dim}_{\C}\;W(\lambda,\mathbf{I})^{\dag}$ does not depend on $\mathbf{I}$ and is given by Theorem \ref{dimension formula}.
\end{cor}
\begin{proof}
If $\lambda=\sum_{i=1}^{l-1}m_i\varpi_i^{\dag}+2m_l\varpi_l^{\dag}$, then we have $$\mathrm{dim}_{\C}\;W(\lambda)_{loc}^{\dag}\geq\mathrm{dim}_{\C}\;W(\lambda,\mathbf{I})^{\dag}\geq \prod_{i=1}^{l}\binom{2l+1}{i}^{m_i}$$ by Proposition \ref{dimension inequality}. From Theorem $\ref{dimension formula}$ (2), this inequality is actually equality. If $\lambda=\sum_{i=1}^{l-1}m_i\varpi_i^{\dag}+(2k-1)\varpi_l^{\dag}$, then we have $$\mathrm{dim}_{\C}\;W(\lambda)_{loc}^{\dag}\geq\mathrm{dim}_{\C}\;W(\lambda,\mathbf{I})^{\dag}\geq\mathrm{dim}_{\C}\;W(\varpi_{l})_{loc}^{\dag}\left(\prod_{i=1}^{l-1}\binom{2l+1}{i}^{\lambda(m_i)}\right)\binom{2l+1}{l}^{k-1}$$ by Proposition \ref{dimension inequality}. From Theorem $\ref{dimension formula}$ (1), this inequality is actually equality. Hence the assertion follows.
\end{proof}
\subsection{Freeness of $W(\lambda)^{\dag}$ over $\mathbf{A}_{\lambda}$}
In this subsection, we prove the following theorem
\begin{thm}\label{freeness}
For each $\lambda \in \mathring{P}^{\dag}_{+}$, the global Weyl module $W(\lambda)^{\dag}$ is free over $\mathbf{A}_{\lambda}$.
\end{thm}

To prove this theorem, we need the following preparatory result:
\begin{thm}\label{polynomialring}
For each $\lambda \in \mathring{P}^{\dag}_{+}$, the algebra $\mathbf{A}_{\lambda}$ is isomorphic to $\mathbf{A}^{\prime}_{\lambda}$.
\end{thm}
Theorem $\ref{polynomialring}$ and Corollary $\ref{dimension equality}$ imply Theorem $\ref{freeness}$ by $\cite{Sus, Qui}$. We prove Theorem $\ref{freeness}$ after proving Theorem $\ref{polynomialring}.$
\begin{proof}[Proof of Theorem \ref{polynomialring}]
We show that the surjection $\mathbf{A}_{\lambda}^{\prime}\to\mathbf{A}_{\lambda}$ is the isomorphism. We have $\mathrm{dim}_{\C}\;\mathbf{A}_{\varpi_l^{\dag}}^{\prime}=1.$ Since $\mathrm{dim}_{\C}\;\mathbf{A}_{\varpi_l^{\dag}}\geq1.$ Hence $\mathbf{A}_{\varpi_l^{\dag}}^{\prime}\to\mathbf{A}_{\varpi_{l}^{\dag}}$ is the isomorphism. If $\lambda\in \mathring{P}_{+}^{\dag\prime},$ then the assertion is Theorem \ref{chari}. We prove the assertion for $\lambda=\sum_{i=1}^{l-1}m_i\varpi_i^{\dag}+(2m+1)\varpi_{l}^{\dag}$. Let $f\in \mathbf{A}_{\lambda}^{\prime}$ be a nonzero element. It is suffice to show that there exists a quotient of $W(\lambda)^{\dag}$ such that $f$ acts nontrivially on the image of the cyclic vector $v_{\lambda}$ of $W(\lambda)^{\dag}.$ Let $\mu:=\lambda-\varpi_l^{\dag}$. We have $\mathbf{A}_{\lambda}^{\prime}\cong\mathbf{A}_{\mu}^{\prime}.$ By checking the defining relations, we have a homomorphism of $\Cg^{\dag}$-module $$W(\lambda)^{\dag}\to W(\varpi_{l}^{\dag})^{\dag}\otimes_{\C} W(\mu)^{\dag}$$ which maps $v_{\lambda}$ to $v_{\varpi^{\dag}_l}\otimes v_{\mu}$. By Theorem \ref{chari}, we have a quotient module $V$ of $W(\mu)^{\dag}$ such that $f$ acts nontrivially on the image of $v_{\mu}\in W(\mu)^{\dag}.$ We have a homomorphism $$W(\lambda)^{\dag}\to W(\varpi_{l}^{\dag})^{\dag}\otimes_{\C} V \to W(\varpi_{l}^{\dag})^{\dag}_{loc}\otimes_{\C} V.$$ Let $v\in V$ and $w_{\varpi_l^{\dag}}\in W(\varpi_l^{\dag})_{loc}^{\dag}$ be the image of $v_{\mu}$ in $V$ and the image of $v_{\varpi_l^{\dag}}$ in $W(\varpi_l^{\dag})_{loc}^{\dag},$ respectively. By Lemma $\ref{p}$, we have $p_{i,r}(w_{\varpi^{\dag}}\otimes v)=w_{\varpi^{\dag}}\otimes p_{i,r}(v)$ for each $i\in \{1,...,l\}$ and $r\in \mathbb{Z_+}.$ Therefore, $f$ acts nontrivially on the highest weight vector $w_{\varpi^{\dag}}\otimes v$ of $W(\varpi_{l}^{\dag})^{\dag}_{loc}\otimes V$. Hence $fv_{\lambda}\neq 0$. Hence the assertion follows. 
\end{proof}

\begin{proof}[Proof of Theorem \ref{freeness}]
We set $N:=\mathrm{dim}\;W(\lambda)^{\dag}_{loc}.$ Let $\mathfrak{m}$ be a maximal ideal of $\mathbf{A}_{\lambda}$. By Nakayama's lemma $\cite[\mathrm{Lemma\; 1.M}]{Mat}$, there exists $f\notin \mathfrak{m}$ such that $(W(\lambda)^{\dag})_{f}$ is generated by N elements as $(\mathbf{A}_{\lambda})_f$-module, where $(W(\lambda)^{\dag})_{f}$ and $(\mathbf{A}_{\lambda})_f$ are the localization of $W(\lambda)^{\dag}$ and $\mathbf{A}_{\lambda}$ by $f$, respectively. Since $(\mathbf{A}_{\lambda})_f$ is Noetherian, we have an exact sequence $(\mathbf{A}_{\lambda})_f^{\oplus M}\overset{\phi}{\to} (\mathbf{A}_{\lambda})_f^{\oplus N}\overset{\psi}{\to} (W(\lambda)^{\dag})_f\to 0.$ For any maximal ideal $\mathfrak{n}$ such that $f\notin \mathfrak{n}$, the induced morphism $\overline{\psi}: (\mathbf{A}_{\lambda})_f^{\oplus N}/\mathfrak{n}(\mathbf{A}_{\lambda})_f ^{\oplus N}\to  (W(\lambda)^{\dag})_f/\mathfrak{n} (W(\lambda)^{\dag})_f$ is an isomorphism by Corollary \ref{dimension equality}. This implies the matrix coefficient of $\phi$ is contained in the Jacobson radical of $(\mathbf{A}_{\lambda})_f.$ Since $(\mathbf{A}_{\lambda})_f$ is an integral domain and finitely generated over $\C,$ we deduce $\phi=0.$ It follows that $(W(\lambda))^{\dag})$ is flat over $\mathbf{A}_{\lambda}$ by $\cite{Jot}$. Since $\mathbf{A}_{\lambda}$ is a polynomial ring, $(W(\lambda))^{\dag})$ is the projective $\mathbf{A}_{\lambda}$-module. From $\cite{Qui, Sus}$, a projective module over a polynomial ring is free. Hence the assertion follows. 
\end{proof}

\end{document}